\documentclass[10pt]{amsproc}
\usepackage[usenames,dvipsnames,svgnames,table]{xcolor}
\usepackage[pagebackref,linktocpage=true,colorlinks=true,linkcolor=Blue,citecolor=BrickRed,urlcolor=RoyalBlue]{hyperref}
\usepackage[alphabetic]{amsrefs}

\usepackage[pagebackref,linktocpage=true,colorlinks=true,linkcolor=Blue,citecolor=BrickRed,urlcolor=RoyalBlue]{hyperref}
\usepackage[alphabetic]{amsrefs}

\usepackage{amsmath}
\usepackage{amsthm}
\usepackage{amssymb}
\usepackage{graphicx}
\usepackage{stackrel}
\usepackage{fullpage}
\usepackage{pdfpages}
\usepackage{tikz}
\usetikzlibrary{patterns}
\usepackage{array}

\def\R{{{\mathbb R}}}

\def\S{{{\mathbb S}}}

\def\normal{{{\vec{n}}}}
\def\d{{{\partial}}}
\def\tr{{{\mathrm{tr}}}}

\newcommand{\inprod}[2]{\left\langle #1 \, , \, #2 \right\rangle}

\newcommand{\eqdef}{\stackrel{{\rm{def}}}{=}}

\newcommand\diam{{\mathrm{diam}}}
\newcommand\dist{{\mathrm{dist}}}

\newtheorem{theorem}{Theorem}[section]
\newtheorem{lemma}[theorem]{Lemma}
\newtheorem{proposition}[theorem]{Proposition}
\newtheorem{corollary}[theorem]{Corollary}

\newtheorem{example}[theorem]{Example}
\newtheorem{condition}[theorem]{Condition}
\newtheorem{remark}[theorem]{Remark}

\numberwithin{equation}{section}

\title{Boundary Estimates for Solutions of the Monge-Amp\`ere Equation Satisfying Dirichlet-Neumann Type Conditions in Annular Domains}
\author{Tim Espin}
\author{Aram Karakhanyan}
\address{School of Mathematics
University of Edinburgh
James Clerk Maxwell Building
The King's Buildings
Peter Guthrie Tait Road
EDINBURGH
EH9 3FD}
\email{Tim.Espin@ed.ac.uk}
\email{aram6k@gmail.com}

\thanks{Tim Espin was supported by The Maxwell Institute Graduate School in Analysis and its Applications, a Centre for Doctoral Training funded by the UK Engineering and Physical Sciences Research Council (grant EP/L016508/01), the Scottish Funding Council, Heriot-Watt University and the University of Edinburgh.}
\thanks{AK was supported by  EPSRC grant 
EP/S03157X/1.}
\begin{document}

\begin{abstract}
	We consider smooth solutions of the Monge-Amp\`ere equation on an annular domain, whose boundary consists of two smooth, closed, strictly convex hypersurfaces, subject to mixed boundary conditions. In particular we impose a homogeneous Dirichlet condition on the outer boundary and a Neumann condition on the inner boundary. We demonstrate that in general, global $C^2$ estimates cannot be obtained unless we impose extra restrictions on the principal curvatures of the inner boundary and on the Neumann condition itself. The latter is illustrated by the construction of an explicit counterexample. Under these conditions, we prove a priori $C^2$ estimates and show that our problem admits a smooth solution.
\end{abstract}

\maketitle

\section{Introduction}

In this article we study a series of Monge-Amp\`ere-type equations satisfying mixed Dirichlet-Neumann boundary conditions on annular domains. Let $\Omega\subset\R^n$ be an annular domain such that $\partial \Omega = \Gamma^+\cup \Gamma^-$, where 
$\Gamma^+$ and $\Gamma^-$ are the exterior and interior components of the boundary respectively. The homogeneous Dirichlet data is prescribed on $\Gamma^+$ whereas on  $\Gamma^-$, $u$ satisfies a Neumann-type condition. Throughout, we assume $\Gamma^+$ and $\Gamma^-$ are disjoint, smooth, closed, strictly convex hypersurfaces enclosing the origin.

The Monge-Amp\`ere equation is one of the central equations 
in nonlinear PDEs, see \cite{TW-Scuola}. Existence of solutions for the Dirichlet problem on a smooth, strictly convex domain with smooth boundary data was established by Caffarelli, Nirenberg and Spruck in \cite{CNS-1} under the assumption that the problem admits a convex subsolution. The Dirichlet problem on more general domains was dealt with by Guan and Spruck in \cite{GuanSpruck,Guan1998}. In \cite{LTU}, Lions, Trudinger and Urbas studied the Neumann problem for the Monge-Amp\`ere equation on a convex domain and proved boundary estimates under very general conditions. Building on this, in \cite{Urbas-95} and \cite{Urbas-98}, Urbas examined the oblique derivative problem, which is more general than the Neumann problem. See also \cite{Ma-19} for more recent results on related Hessian equations.


\subsection{The Basic Monge-Amp\`ere Equation}

We first investigate the following Monge-Amp\`{e}re equation on an annular domain $\Omega \subset \R^n$:
\begin{equation}\label{TheEquation}
\begin{cases}
\det[D^2 u] = \psi^n (x)  & \text{ in } \Omega \, , \\
u = 0  & \text{ on } \Gamma^+ \, , \\
u_\nu = \gamma_0 u + \phi(x) & \text{ on } \Gamma^- \, ,
\end{cases}
\end{equation}
where $\nu$ is the inward-pointing (into $\Omega$) unit normal vector field on $\Gamma^-$. We assume that $\gamma_0 \geq 0$ is a given constant, and $\psi(x) > 0$ and $\phi(x)$ are smooth functions on $\Omega$ and $\Gamma^-$ respectively.

To establish boundary $C^2$ estimates for the Dirichlet problem for the Monge-Amp\`ere equation one has to bound the mixed tangential-normal derivatives on the boundary, see \cite{CNS-1} or Chapter 17 in \cite{GilbargTrudinger}. 
For the Neumann problem, these follow from the Neumann condition on the boundary.

Here, we consider the problem with both Dirichlet and Neumann conditions. 	
The main issue is to estimate the second order tangential and normal derivatives on $\Gamma^-$.
This is usually done  by assuming $\gamma_0$ to be a large constant, as in \cite{LTU}. 	
It was shown in \cite{AK} that there exists a weak solution of \eqref{TheEquation} when $\Gamma^-$ is ``free" and, in addition to $u_\nu = \gamma_0 u + \phi$,  $u$ is also constant on $\Gamma^-$.
However, we show here that for general $\Omega$ (when $\Gamma^-$ is fixed) the solution of \eqref{TheEquation} may not exist even when $\Omega$ is radially symmetric and $u$ is constant on $\Gamma^-$. See Section \ref{sec-explicit} for details.

However, if a solution exists then it is interesting to know  whether one can establish uniform estimates for smooth solutions. It turns out that second derivative estimates on $\Gamma^-$ can be obtained only for the double tangential derivatives, provided that the principal curvatures of $\Gamma^-$ are not too large. In order to bound the double normal derivative, one must additionally require the existence of a subsolution whose normal derivative on $\Gamma^-$ satisfies a particular second-order ODE. This assumption is not too unreasonable: indeed, \cite{CNS-1,GuanSpruck,Guan1998} all assume the existence of a subsolution to obtain estimates for the Dirichlet problem.

In this direction, our main results and conditions are the following:

\begin{theorem}[$C^0$ and $C^1$ Boundedness of Solutions]\label{main-th-C01}
	Suppose $u\in C^2(\Omega)\cap C^1(\overline{\Omega})$ is a strictly convex solution of \eqref{TheEquation} in the annular domain $\Omega$, such that $\Gamma^\pm$ are smooth, closed, strictly convex hypersurfaces.
	Then there are constants $C_0,C_1$ depending on $n, \Omega, \psi$ and $\phi$ such that \[
	\sup_{\overline{\Omega}} |u| \leq C_0  \, , \quad \|u\|_{C^1(\overline{\Omega})}\leq C_1 \, .
	\]
\end{theorem}

\begin{condition}[Curvature Condition on $\Gamma^-$]\label{CurvatureCondition}
	Let $\xi$ be a unit vector tangential to $\Gamma^-$. Suppose that $\gamma_0 > 0$ and the normal curvature $\kappa_{\xi}$ in the direction $\xi$ is bounded above such that
	\begin{equation*}
	2\kappa_\xi < \gamma_0 + \max \left\{ 0 , \min_{\Gamma^-}\frac{\gamma_0 u + \phi}{M-u} \right\} \, ,
	\end{equation*}
	for some (large) constant $M$ depending on $n, \Omega, \psi, \phi$ and $C_1$.
\end{condition}

\begin{theorem}[Boundedness of Double Tangential Derivatives on $\Gamma^-$ under the Curvature Condition]\label{main-th-tan}
	Suppose $u\in C^4(\Omega)\cap C^3(\overline{\Omega})$ is a strictly convex solution of \eqref{TheEquation}, and let $\Omega$ and $\Gamma^\pm$ be as in Theorem \ref{main-th-C01}. Suppose that Condition \ref{CurvatureCondition} holds for all tangential unit vectors $\xi$ at each point on $\Gamma^-$. Then \[
	\sup_{\Gamma^-} |D_{\xi\xi} u| \leq C_2 \, ,
	\]
	where $D_{\xi\xi}u$ is the second derivative of $u$ in the direction $\xi$ and $C_2$ depends on $M$ and $\gamma_0$.
\end{theorem}

Moreover, by differentiating the Neumann condition, \[
\sup_{\Gamma^-} |D_{\xi\nu} u| \leq C_3
\]

\begin{condition}[Subsolution Condition]\label{SubsolutionCondition}
	Suppose there exists a strictly convex subsolution $\underline{u} \in C^4(\Omega)\cap C^3(\overline{\Omega})$ satisfying
	\begin{equation}\label{TheEquationSubsolution}
	\begin{cases}
	\det[D^2 \underline{u}] \geq \psi^n  & \text{ in } \Omega \, , \\
	\underline{u} = 0  & \text{ on } \Gamma^+ \, , \\
	\underline{u}_\nu = \gamma_0 \underline{u} + \phi(x) & \text{ on } \Gamma^- \, ,
	\end{cases}
	\end{equation}
	with the same $\phi$ as in \eqref{TheEquation}. Suppose additionally that $\gamma_0 > 0$ and that there exists $\tau > 0$ depending on $n, \Omega, \psi$ and $\phi$ (but not $\underline{u}$) such that for all $x_0\in\Gamma^-$, any unit vector $\xi$ tangential to $\Gamma^-$ at $x_0$ and any unit-speed geodesic curve $\gamma : (-a,a) \longrightarrow \Gamma^-$ passing through $x_0$ with the properties
	\begin{enumerate}
		\item[(i)] $\gamma(0) = x_0$,
		\item[(ii)] $\gamma'(s)|_{s=0} = \xi$,
		\item[(iii)] $\gamma''(s)|_{s=0} = -\kappa_{\xi} \nu$,
	\end{enumerate}
	the differential inequality \[
	\frac{1}{\gamma_0} \frac{d^2}{ds^2} U + \kappa_\xi (x_0) U + \underline{u}_{\xi\xi}(x_0) \geq \tau
	\] holds at $x_0$. Here $U(s) = (u_\nu - \underline{u}_\nu)(\gamma(s))$ and $\kappa_\xi(x_0)$ is the normal curvature of $\Gamma^-$ at $x_0$ in the direction $\xi$.
\end{condition}

\begin{theorem}[Boundedness of Double Normal Derivative on $\Gamma^-$ under the Subsolution Condition]\label{main-th-nor}
	Suppose $u\in C^4(\Omega)\cap C^3(\overline{\Omega})$ is a strictly convex solution of \eqref{TheEquation}, and let $\Omega$ and $\Gamma^\pm$ be as in Theorem \ref{main-th-C01}. Suppose that Condition \ref{SubsolutionCondition} holds on $\Gamma^-$. Then \[
	\sup_{\Gamma^-} |D_{\nu\nu}| \leq C_4 \, .
	\]
\end{theorem}

The estimates of second derivatives on $\Gamma^+$ follow from previous results on the Dirichlet problem \cite{CNS-1}. The boundary estimates of Theorems \ref{main-th-tan} and \ref{main-th-nor} then imply the uniform global estimate \[
\|u\|_{C^2(\overline{\Omega})} \leq C
\] for solutions $u\in C^4(\Omega)\cap C^3(\overline{\Omega})$ of \eqref{TheEquation}. Via the method of continuity, this a priori estimate in turn implies the following:

\begin{corollary}[Existence of Solutions]\label{Cor-Existence}
	Under Conditions \ref{CurvatureCondition} and \ref{SubsolutionCondition}, there exists a strictly convex solution $u\in C^{2,\alpha}$ of \eqref{TheEquation}.
\end{corollary}


\subsection{Generalisations and Applications}

Next, we consider a Monge-Amp\`ere-type equation with more general right-hand side:
\begin{equation}\label{TheEquationDu}
\begin{cases}
\det[D^2 u] = \psi^n (x,u,Du)  & \text{ in } \Omega \, , \\
u = 0  & \text{ on } \Gamma^+ \, , \\
u_\nu = \gamma_0 u + \phi(x) &\text{ on } \Gamma^- \, ,
\end{cases}
\end{equation}
where $\Omega, \Gamma^+, \Gamma^-, \nu, \gamma_0$ and $\phi$ are as before, but now the function $\psi$ potentially depends additionally on $u$ and $Du$. Further, we assume that $\psi_z(x,z,p) \geq 0$.

This is a natural generalisation of \eqref{TheEquation} and was also studied by \cite{LTU} in the case of convex domains with Neumann boundary conditions. The Dirichlet problem is treated in such works as \cite{CNS-1}, where the solvability is reduced to the existence of smooth subsolutions, \cite{TUDirichletPGC}, where an almost-sharp structure condition is introduced to bound the gradient of solutions, and Chapter 17 of \cite{GilbargTrudinger}. Obtaining a priori estimates for equations of this type requires extra conditions.

In order to estimate $|u|$, we impose the well-known structure conditions (see for example \cite{GilbargTrudinger}, \cite{LTU}) on $\psi$:

\begin{condition}[Structure Conditions]\label{StructureConditions}
	Let $\psi = \psi(x,z,p)$, $x\in\Omega$, $z \leq 0$, $p \in \R^n$. Suppose that there exist positive functions $g:\Omega \longrightarrow (0,\infty)$ and $h:\R^n \longrightarrow (0,\infty)$ such that
	\begin{equation}\label{Structure}
	\psi^n (x,z,p) \leq \frac{g(x)}{h(p)} \quad \forall x\in\Omega, \ z \leq 0, \ p \in \R^n \, ,
	\end{equation}
	such that $g \in L^1(\Omega)$, $h \in L^1_\text{loc}(\R^n)$, and
	\begin{equation}\label{StructureIneq}
	\int_\Omega g < \int_{\R^n} h \, .
	\end{equation}
	Further, we define the constant $R_0$ such that
	\begin{equation}\label{R0}
	\int_\Omega g = \int_{|p|\leq R_0} h \, .
	\end{equation}
\end{condition}

However, we construct an example showing that this is not sufficient for an a priori estimate on the gradient of solutions. Without extra conditions, at best only a local gradient estimate can be obtained. For a global estimate, we use condition (17.77) from Chapter 17 of \cite{GilbargTrudinger}, which is a more general form of the condition studied in \cite{TUDirichletPGC}.

\begin{condition}[Further Structure Condition]\label{StructureConditionGradient}
	Let $\mathcal{N}$ be a neighbourhood of $\Gamma^+$ and for $x\in\mathcal{N}$ let $d_x = \dist(x,\Gamma^+)$. Suppose that there exist $\beta,Z$ with $\beta \geq 0$ and $Z$ a nondecreasing function such that
	\begin{equation}\label{StructureGradient}
	\psi^n(x,z,p) \leq Z(|z|)d_x^\beta |p|^{\beta + n + 1} \, ,
	\end{equation}
	for all $x\in\mathcal{N}$, $z\in\R$ and $|p|\geq Z(|z|)$.
\end{condition}

For the global gradient bound, we could also assume the existence of a subsolution. It is natural to assume the existence of a subsolution satisfying \eqref{TheEquationSubsolution} since this is needed for the full $C^2$ estimate. Of course Condition \ref{SubsolutionCondition} is more than sufficient for this purpose.

Interestingly, due to the form of barrier we use to prove the tangential $C^2$ estimates on $\Gamma^-$, we actually require the normal $C^2$ estimate on $\Gamma^-$ to be proved first, and thus the tangential estimate also indirectly requires Condition \ref{SubsolutionCondition}.

Specifically, we prove the following results:

\begin{theorem}[$C^0$ and $C^1$ Boundedness of Solutions]\label{main-th-C01Du}
	Suppose $u\in C^2(\Omega)\cap C^1(\overline{\Omega})$ is a strictly convex solution of \eqref{TheEquationDu} in the annular domain $\Omega$, such that $\Gamma^\pm$ are smooth, closed, strictly convex hypersurfaces. Suppose that Condition \ref{StructureConditions} holds. Then there exists a constant $C_0'$ depending on $n, \Omega, \psi$ and $\phi$ such that \[
	\sup_{\overline{\Omega}} |u| \leq C_0'  \, .
	\] Furthermore, if either $\psi = \psi(x,u)$ only and Condition \ref{StructureConditions} holds, or $\psi = \psi(x,u,Du)$ and Conditions \ref{StructureConditions} and \ref{StructureConditionGradient} hold,
	then there exists a constant $C_1'$ depending on $n, \Omega, \psi$ and $\phi$ such that \[
	\|u\|_{C^1(\overline{\Omega})} \leq C_1'  \, .
	\]
\end{theorem}

\begin{condition}
	[Curvature Condition on $\Gamma^-$]\label{CurvatureConditionDu}
	Let $\xi$ be a unit vector tangential to $\Gamma^-$. Suppose that $\gamma_0 > 0$ and the normal curvature $\kappa_{\xi}$ in the direction $\xi$ is bounded above such that
	\begin{equation*}
	2\kappa_\xi + \tilde{C} < \gamma_0 + \max \left\{ 0 , \min_{\Gamma^-}\frac{\gamma_0 u + \phi}{\tilde{M} + (1-N^4)u} \right\} \, ,
	\end{equation*}
	where $N = N(n,\Omega,\psi,\phi)$, $\tilde{C} = \tilde{C}(n,\Omega,\Gamma^\pm,\psi,\phi,N)$ and $\tilde{M} = \tilde{M}(n,\Omega,\Gamma^\pm,\psi,\phi)$ are constants.
\end{condition}

\begin{theorem}[$C^2$ Boundedness of Solutions]\label{main-th-C2Du}
	Suppose $u \in C^4(\Omega)\cap C^3(\overline{\Omega})$ is a strictly convex solution of \eqref{TheEquationDu}, and let $\Omega$ and $\Gamma^\pm$ be as in Theorem \ref{main-th-C01Du}. Suppose that either $\psi = \psi(x,u)$ only and Condition \ref{CurvatureCondition} holds, or $\psi = \psi(x,u,Du)$ and Condition \ref{CurvatureConditionDu} holds. Suppose further that there exists a subsolution satisfying Condition \ref{SubsolutionCondition}. Then there exists a constant $C_2' = C_2'(n,\Omega,\Gamma^\pm,\psi,\gamma_0,\phi)$ such that \[
	\sup_{\overline{\Omega}}|D^2 u| \leq C_2' \, .
	\]
\end{theorem}

As with the $\psi = \psi(x)$ case, the estimates of second derivatives on $\Gamma^+$ follow from previous results on the Dirichlet problem \cite{CNS-1}. Via the method of continuity, these a priori estimates in turn imply the existence of solutions.

\begin{corollary}[Existence of Solutions]\label{Cor-ExistenceDu}
	Under the conditions of Theorems \ref{main-th-C01Du} and \ref{main-th-C2Du}, there exists a strictly convex solution $u\in C^{2,\alpha}$ of \eqref{TheEquationDu}.
\end{corollary}

We are also interested in applications of these results to the equation of prescribed Gauss curvature,

\begin{equation}\label{PrescribedGauss}
\begin{cases}
\det[D^2u] = K(x)(1 + |Du|^2)^{\frac{n+2}{2}}  & \text{ in } \Omega \, , \\
u = 0  & \text{ on } \Gamma^+ \, , \\
u_\nu = \gamma_0 u + \phi(x) & \text{ on } \Gamma^- \, .
\end{cases}
\end{equation}

This is merely an equation of the same type as \eqref{TheEquationDu}. We insist that the Gauss curvature $K$ satisfies two conditions:
\begin{condition}\label{PGCondition}
	Let $\omega_n$ be the volume of the unit $n$-ball in $\R^n$. Suppose $K$ is a smooth function such that
	\begin{equation}\label{GClessthannball}
	\int_{\Omega} K < \omega_n \, ,
	\end{equation}
	\begin{equation}\label{K=0}
	K(x) = 0 \quad \forall x\in\Gamma^+ \, .
	\end{equation}
\end{condition}
It has been shown many times (see for example \cite{UrbasWithoutBC}) that \eqref{GClessthannball} with a nonstrict inequality is a necessary condition for the existence of a convex, $C^2$ solution of \eqref{PrescribedGauss}. This follows from the fact that the gradient map $Du : \Omega \longrightarrow \R^n$ is injective when $u$ is convex, and has Jacobian $\det[D^2 u]$. Therefore by integrating \eqref{PrescribedGauss}, \[
\int_{\Omega} K = \int_{Du(\Omega)} \frac{\mathrm{d}p}{(1 + |p|^2)^{(n+2)/2}} \leq \int_{\R^n} \frac{\mathrm{d}p}{(1 + |p|^2)^{(n+2)/2}} = \omega_n \, .
\] However, Condition \ref{PGCondition} is also sufficient for proving $C^0$ and global $C^1$ bounds for solutions of \eqref{PrescribedGauss}.

We shall also see in this case that thanks to the structure of the right-hand side of \eqref{PrescribedGauss}, namely $\psi$ is convex with respect to $Du$, we may relax Condition \ref{CurvatureConditionDu} on the curvatures of $\Gamma^-$ to prove the $C^2$ estimate. In fact, we only need assume Conditions \ref{CurvatureCondition} and \ref{SubsolutionCondition} hold, as in the case where $\psi = \psi(x)$. Therefore, provided Conditions \ref{PGCondition}, \ref{CurvatureCondition} and \ref{SubsolutionCondition} hold, there exists a convex, $C^{2,\alpha}$ solution of \eqref{PrescribedGauss}.


\subsection{Parabolic Flows}

Finally, we generalise our results to parabolic flow equations. Let $0 < T < \infty$, and define the open parabolic cylinder $\Omega_T \eqdef \Omega\times(0,T)$. Since $\Omega$ is annular, $\Omega_T$ is not really a cylinder. In fact, if $\Omega$ is a planar domain, it looks like a tall, somewhat angular doughnut. We settle on referring to $\Omega_T$ as an ``annular cylinder". Consider the equation
\begin{equation}\label{TheFlowEquation}
\begin{cases}
-u_t \det[D^2u] = \psi^n(x,u,Du)  & \text{ in } \Omega_T \, , \\
u(x,0) = u_0(x) & \text{ at } t=0 \, , \\
u(\cdot,t) = \vartheta(t) & \text{ on } \Gamma^+ \, ,\\
u_\nu(\cdot,t) = \gamma_0 u(\cdot,t) + \phi(\cdot,t) & \text{ on } \Gamma^- \, ,
\end{cases}
\end{equation}
where $u_0\in C^\infty(\overline{\Omega})$ is a strictly convex solution of $\det[D^2u_0] = \psi^n(x,u_0,Du_0)$ on $\Omega$, $\vartheta'(t) < 0$ for all $t$ and for each fixed $x\in\Gamma^-$,  $\phi_t(x,t) > 0$ for all $t$. As usual, we assume additionally that $\psi_z(x,z,p) \geq 0$.

There exist several different forms for parabolic versions of the Monge-Amp\`ere equation. A discussion of some of these may be found in the notes appended to Chapter 15 of \cite{Lieberman}. For instance, the equation $-u_t + (\det[D^2 u])^{1/n} = \psi$ was studied by Ivochkina and Ladyzhenskaya in \cite{IvoLady1}, where they proved that the Cauchy-Dirichlet problem for this equation has a convex solution in a cylindrical domain under some conditions. They also extended these results to more general equations involving other symmetric polynomials of the eigenvalues of the Hessian, where $(\det[D^2 u])^{1/n}$ is replaced by $S_k^{1/k}$, in \cite{IvoLady2}.

Equations of the form \eqref{TheFlowEquation} are related to the equations of motion for surfaces evolving under various versions of the Gauss curvature flow. In \cite{Tso}, it is shown that the support function of a surface moving by its Gauss curvature $K$ satisfies \eqref{TheFlowEquation} with $\psi^n(x,u,Du) = (1 + |x|^2)^{-1/2(n+1)}$. We will instead link \eqref{TheFlowEquation} with the inverse Gauss curvature flow, whereby the graph of a surface expands at speed $1/K$, without needing to use the support function. \\

Our main results and conditions are:

\begin{theorem}[Upper and Lower Bounds on the Time Derivative of Solutions]\label{main-th-Flowut}
	Suppose $u\in C^2(\Omega_T)\cap C^1(\overline{\Omega_T})$ is a strictly convex solution of \eqref{TheFlowEquation} with $\psi_z \geq 0$ in the annular cylinder $\Omega_T$, such that $\Gamma^\pm$ are smooth, closed, strictly convex hypersurfaces. Then there exist constants $C_T, C^T > 0$ depending on $T,\vartheta,\gamma_0$ and $\phi$ such that
	\begin{equation}
	C_T \leq |u_t(x,t)| \leq C^T \, ,
	\end{equation}
	for all $(x,t)\in\overline{\Omega}\times[0,T]$.
\end{theorem}

A quirk of the form of these parabolic equations is that on a finite time interval, no extra conditions are required for the estimate on $|u|$ besides those needed to ensure $u_0$ is bounded, even in the full generality of $\psi = \psi(x,u,Du)$. On the other hand, analogously to Condition \ref{StructureConditionGradient} for the elliptic case, Condition (15.46) in \cite{Lieberman} is needed for a global gradient estimate when $\psi = \psi(x,u,Du)$.

\begin{condition}[Structure Condition]\label{StructureConditionFlow}
	Let $\mathcal{N}$ be a neighbourhood of $\Gamma^+$ and for $x\in\mathcal{N}$ let $d_x = \dist(x,\Gamma^+)$. Suppose that there exist $\beta,Z$ with $\beta \geq 0$ and $Z$ a nondecreasing function such that
	\begin{equation*}
	\psi^n(x,z,p) \leq Z(|z|)d_x^\beta |p|^{\beta + n + 2} \, ,
	\end{equation*}
	for all $x\in\mathcal{N}$, $z\in\R$ and $|p|\geq Z(|z|)$.
\end{condition}

\begin{theorem}[$C^0$ and $C^1$ Boundedness of Solutions]\label{main-th-FlowC01}
	Suppose $u\in C^2(\Omega_T)\cap C^1(\overline{\Omega_T})$ is a strictly convex solution of \eqref{TheFlowEquation}, and let $\Omega_T$ and $\Gamma^\pm$ be as in Theorem \ref{main-th-Flowut}. Then there exists a constant $C_0^T$ depending on $T,u_0,\vartheta,\gamma_0$ and $\phi$ such that \[
	\sup_{\overline{\Omega_T}} |u| \leq C_0^T  \, .
	\] Furthermore, if either $\psi = \psi(x,u)$ only, or $\psi = \psi(x,u,Du)$ and Condition \ref{StructureConditionFlow} holds, then there exists a constant $C_1^T$ depending on $n, \Omega, \psi, \phi$ and $T$ such that \[
	\|u\|_{C^1(\overline{\Omega_T})} \leq C_1^T \, .
	\]
\end{theorem}

\begin{condition}[Strict Parabolic Subsolution Condition]\label{FlowSubsolutionCondition}
	Suppose there exists a strictly convex subsolution $\underline{u} \in C^4(\Omega_T)\cap C^3(\overline{\Omega_T})$ satisfying
	\begin{equation}\label{TheFlowEquationSubsolution}
	\begin{cases}
	-u_t \det[D^2 \underline{u}] \geq \psi^n(x,\underline{u},D\underline{u}) + \delta_0  & \text{ in } \Omega_T \, , \\
	\underline{u}(x,0) = \underline{u}_0(x) & \text{ at } t=0 \, , \\
	\underline{u}(\cdot,t) = \vartheta(t)  & \text{ on } \Gamma^+ \, , \\
	\underline{u}_\nu(\cdot,t) = \gamma_0 \underline{u}(\cdot,t) + \phi(\cdot,t) & \text{ on } \Gamma^- \, ,
	\end{cases}
	\end{equation}
	for some small $\delta_0 > 0$, where $\underline{u}_0$ is a subsolution satisfying \eqref{TheEquationSubsolution}. Suppose additionally that $\gamma_0 > 0$ and that at each $t\in[0,T]$ there exists $\tau > 0$ depending on $n, \Omega, \psi$ and $\phi$ (but not $\underline{u}$ or $t$) such that for all $x_0\in\Gamma^-$, any unit vector $\xi$ tangential to $\Gamma^-$ at $x_0$ and any unit-speed geodesic curve $\gamma : (-a,a) \longrightarrow \Gamma^-$ passing through $x_0$ with the properties
	\begin{enumerate}
		\item[(i)] $\gamma(0) = x_0$,
		\item[(ii)] $\gamma'(s)|_{s=0} = \xi$,
		\item[(iii)] $\gamma''(s)|_{s=0} = -\kappa_{\xi} \nu$,
	\end{enumerate}
	the differential inequality \[
	\frac{1}{\gamma_0} \frac{d^2}{ds^2} U + \kappa_\xi (x_0) U + \underline{u}_{\xi\xi}(x_0) \geq \tau
	\] holds at $x_0$. Here $U(s) = (u_\nu - \underline{u}_\nu)(\gamma(s))$ and $\kappa_\xi(x_0)$ is the normal curvature of $\Gamma^-$ at $x_0$ in the direction $\xi$. Note that by the parabolic maximum principle (see Theorem 14.1 of \cite{Lieberman}), such a subsolution remains beneath a solution of \eqref{TheFlowEquation} for all time.
\end{condition}

\begin{theorem}[$C^2$ Boundedness of Solutions]\label{main-th-FlowC2}
	Suppose $u \in C^4(\Omega_T)\cap C^3(\overline{\Omega_T})$ is a strictly convex solution of \eqref{TheFlowEquation}, and let $\Omega_T$ and $\Gamma^\pm$ be as in Theorem \ref{main-th-Flowut}. Suppose that either $\psi = \psi(x,u)$ and Condition \ref{CurvatureCondition} holds on $\Gamma^-$, or $\psi = \psi(x,u,Du)$ and Condition \ref{CurvatureConditionDu} holds. Suppose further that there exists a subsolution satisfying Condition \ref{FlowSubsolutionCondition}. Then there exists a constant $C_2^T = C_2^T(n,\Omega,\Gamma^\pm,\psi,\gamma_0,\phi,T)$ such that \[
	\sup_{\overline{\Omega}}|D^2 u| \leq C_2^T \, .
	\]
\end{theorem}

These a priori estimates together imply the existence of solutions.

\begin{corollary}[Existence of Solutions]\label{Cor-FlowExistence}
	Under the conditions of Theorems \ref{main-th-FlowC01} and \ref{main-th-FlowC2}, there exists a strictly convex solution $u\in C^{2,\alpha}$ of \eqref{TheFlowEquation}.
\end{corollary}


\subsection{Outline of the Article}

The outline of this article is as follows. Sections \ref{sec-C0} and \ref{sec-C1} are concerned with proving the estimates of Theorem \ref{main-th-C01}. These follow standard proofs and also apply if $\gamma_0 = 0$. The first of our second derivative estimates, Theorem \ref{main-th-tan}, is proved in Section \ref{sec-tan}. This employs a barrier-type argument similar to \cite{LTU}, and the curvature condition \ref{CurvatureCondition} on $\Gamma^-$ is used to complete the proof. To round off this part, in Section \ref{sec-explicit} we investigate some specific solutions of \eqref{TheEquation}, culminating in a counterexample which demonstrates that without extra conditions there exist solutions with arbitrarily large double normal derivative on $\Gamma^-$. This is primarily conducted in the 2-dimensional framework when $\Omega$ is a planar domain, but we also show how the argument extends to the $n$-dimensional setting.

Motivated by these examples, we introduce and explain the Subsolution Condition \ref{SubsolutionCondition} in Section \ref{sec-normal}. We show that under this condition, $\inf_{\Gamma^-} u_{\xi\xi} \geq \tau > 0$ on $\Gamma^-$ for any tangential direction $\xi$. This then implies an estimate on the double normal derivative thanks to the structure of the Monge-Amp\`ere equation. In Section \ref{sec-existence} we apply standard Krylov-Safonov-type estimates and the method of continuity to obtain the existence result of Corollary \ref{Cor-Existence} under the conditions of Theorems \ref{main-th-tan} and \ref{main-th-nor}.

Section \ref{sec-generalRHS} is devoted to generalising the results to Monge-Amp\`ere-type equations where $\psi = \psi(x,u,Du)$. Section \ref{sec-C0Du} covers the $C^0$ estimate under the additional structure condition \ref{StructureConditions}. We present a counterexample in Section \ref{sec-CounterexampleDu} to demonstrate that further conditions are needed for a global gradient estimate. Local and global estimates on the gradient are then carried out in Section \ref{sec-C1Du}. The $C^2$ estimate of Theorem \ref{main-th-C2Du} is carried out in Section \ref{sec-C2Du}, where, alongside our subsolution condition, we must assume a stronger condition on the principal curvatures of $\Gamma^-$ than for the $\psi = \psi(x)$ case.

An application of our estimates to the equation of prescribed Gauss curvature is developed in Section \ref{sec-PrescribedGC}. Estimates for the maximum modulus, gradient and Hessian of solutions are outlined, and in Section \ref{sec-PGCBC} we briefly describe a setting which leads naturally to the Neumann condition on the inner boundary.

In Section \ref{sec-Parabolic}, we generalise our results even further to parabolic Monge-Amp\`ere-type equations. After first deriving bounds on the time derivative of solutions $u_t$ in Section \ref{sec-Flowut}, the usual $C^0, C^1$ and $C^2$ estimates are carried out across sections \ref{sec-FlowC0}, \ref{sec-FlowC1} and \ref{sec-FlowC2}.


\section{The Basic Monge-Amp\`ere Equation with $\psi=\psi(x)$}\label{sec-Basic}

\subsection{The $C^0$ Estimate}\label{sec-C0}
In this section we prove uniform estimates for solutions of the problem \eqref{TheEquation}. Our proof uses a barrier construction and it works for the case $\gamma_0=0$.

\begin{proposition}\label{C0Estimate}
	Suppose $u\in C^2(\Omega)\cap C^1(\overline{\Omega})$ is a strictly convex solution of \eqref{TheEquation}. Then $u$ is bounded, with \[
	0 \leq -u \leq C_0 \eqdef e^{1/2} \max\left\{ \frac{\|\psi\|_\infty}{K(1 - K\max_{\overline{\Omega}}|x|^2)} , \frac{\max_{\Gamma^-}|\phi(x)|}{K m_0} \right\} \, ,
	\] for any constant $K < 1/\max_{\overline{\Omega}} |x|^2$. Here, $m_0 = \min_{x\in\Gamma^-}\inprod{x}{\nu} > 0$.
\end{proposition}

\begin{proof}
	Let $v(x) = \ln [u(x)^2] + K|x|^2$ on $\Omega$, for some positive constant $K$ to be chosen. Note that $v = -\infty$ on $\Gamma^+$, and so $v$ can have a local maximum at a point $x_0$ only in the interior of $\Omega$ or on $\Gamma^-$. We consider these two possibilities separately and show $u$ is bounded in each case. \\
	
	First, suppose $v$ has a local maximum at some point $x_0$ in the interior of $\Omega$. Then at $x_0$, we have
	\begin{align}
	& v_i = \frac{2u_i}{u} + 2Kx_i = 0 \, , \label{C0 vi} \\
	& v_{ii} = \frac{2u_{ii}}{u} - \frac{2u_i^2}{u^2} + 2K \leq 0 \, , \label{C0 vii}
	\end{align}
	for each $1 \leq i \leq n$. By choosing coordinates at $x_0$ such that $D^2 u$ is diagonal, the inverse of $D^2 u$, whose components are denoted $u^{ij}$, satisfies the bound on its trace
	\begin{equation}\label{C0BoundOnTrace}
	\sum_{i=1}^n u^{ii} = \sum_{i=1}^n \frac{1}{u_{ii}} \geq \frac{n}{(\prod_{i=1}^n u_{ii})^{1/n}} = \frac{n}{\psi} \geq \frac{n}{\|\psi\|_\infty} \, .
	\end{equation}
	
	Now \eqref{C0 vii} gives
	\begin{align*}
	0 \geq \sum_{i=1}^n u^{ii}v_{ii} & = \frac{2u^{ii}u_{ii}}{u} - 2u^{ii} \left(\frac{u_i}{u}\right)^2 + 2K \sum_{i=1}^n u^{ii} \\
	& \stackrel{\eqref{C0 vi}}{=} \frac{2n}{u} - 2K^2 \sum_{i=1}^n u^{ii} x_i^2 + 2K \sum_{i=1}^n u^{ii} \\
	& \geq \frac{2n}{u} - 2(K^2 \max_{\overline{\Omega}} |x|^2 - K) \sum_{i=1}^n u^{ii} \\
	& \stackrel{\eqref{C0BoundOnTrace}}{\geq} \frac{2n}{u} + 2K(1 - K\max_{\overline{\Omega}} |x|^2)\frac{n}{\|\psi\|_\infty} \, ,
	\end{align*}
	provided that \[
	K < \frac{1}{\max_{\overline{\Omega}} |x|^2} \, .
	\] 
	
	Observe that this implies $u(x_0) \leq 0$, as expected. Thus if \[
	-u > \frac{\|\psi\|_\infty}{K(1 - K\max_{\overline{\Omega}} |x|^2)} \, ,
	\] then we have $u^{ii}v_{ii} > 0$, which is a contradiction. Consequently, if $v$ has an interior maximum then \[
	-u \leq \frac{\|\psi\|_\infty}{K(1 - K\max_{\overline{\Omega}} |x|^2)} \, .
	\]
	
	Next, suppose that $v$ has a maximum at $x_0 \in \Gamma^-$. Then at $x_0$,
	\begin{equation}\label{HopfC0}
	0 \geq v_\nu = \frac{2u_\nu}{u} + 2K\inprod{x_0}{\nu} = 2\frac{\gamma_0 u(x_0)+\phi(x_0)}{u(x_0)} + 2K\inprod{x_0}{\nu} \, ,
	\end{equation}
	where $\nu$ is the inner normal to $\Gamma^-$. Since $\Gamma^-$ is a strictly convex surface enclosing the origin,
	\begin{equation}\label{Convex}
	\inprod{x_0}{\nu} \geq \min_{x\in\Gamma^-} \inprod{x}{\nu} \eqdef m_0 > 0 \, .
	\end{equation}
	Now, $u(x_0)$ and $\phi(x_0)$ must have opposite signs or else we get a contradiction with \eqref{HopfC0}, so since $u(x_0) < 0$, \[
	-u(x_0) \leq \frac{\phi(x_0)}{ \gamma_0+ K m_0} \leq \frac{\max_{\Gamma^-}|\phi|}{ K m_0} \, .
	\] Thus we have \[
	|u(x_0)| \leq \max\left\{ \frac{\|\psi\|_\infty}{K(1 - K\max_{\overline{\Omega}} |x|^2)} , \frac{\max_{\Gamma^-} |\phi(x)|}{K \min_{x\in\Gamma^-}\inprod{x}{\nu(x)}} \right\} \, ,
	\] for any $K < 1/\max_{\overline{\Omega}} |x|^2$.
	Finally, for any $x\in\Omega$,
	\begin{gather*}
	|u(x)| = \exp\left( \frac{v(x) - K|x|^2}{2} \right) \leq \exp\left( \frac{v(x_0)}{2} \right) = \exp\left( \frac{\ln u(x_0)^2 + K|x_0|^2}{2} \right) \\
	\leq |u(x_0)| \exp\left( \frac{1}{2} \right) \, .
	\end{gather*}
	Hence the result follows.
\end{proof}

\begin{remark} 
	The estimate \eqref{Convex} is true even if $\Gamma^-$ is star shaped with respect to the origin of coordinates. Furthermore, it is independent of the choice of $\gamma_0$.
\end{remark}


\subsection{The $C^1$ Estimate}\label{sec-C1}

In this section we prove a boundary gradient estimate for $u\in C^2(\Omega)\cap  C^1(\overline{\Omega})$ by comparison with the defining function of the strictly convex set enclosed by $\Gamma^+$.

\begin{proposition}\label{Gradient}
	Suppose $u \in C^2(\Omega)\cap C^1(\overline{\Omega})$ is a strictly convex solution of \eqref{TheEquation}. Then
	\begin{equation}\label{GradientBound}
	\sup_{\overline{\Omega}} |Du| \leq C_1 \eqdef \max \left\{ \frac{C_0}{\min_{\Gamma^-}|\rho|} , \frac{\|\psi\|_\infty}{\lambda_\emph{min}} \right\} \max_{\Gamma^+} |D \rho| \, ,
	\end{equation} where $\rho$ is a strictly convex defining function for $\Gamma^+$ such that both $\rho(x) = 0$ and $0 < |D \rho(x)| \leq 1$ on $\Gamma^+$, $\lambda_\emph{min}$ is the minimum eigenvalue of $D^2 \rho$, and $C_0$ is defined in Proposition \ref{C0Estimate}.
\end{proposition}

\begin{proof}
	Since $u$ is convex, the maximum of $|Du|$ occurs on the outer boundary $\Gamma^+$. From the Dirichlet condition, the tangential derivatives of $u$ are zero, and so $|Du| = -u_\normal$, where $\normal$ is the inward-pointing unit normal field on $\Gamma^+$.
	
	Let $\rho$ be a strictly convex defining function for $\Gamma^+$ such that both $\rho(x) = 0$ and $0 < |D \rho(x)| \leq 1$ on $\Gamma^+$. Consider the function \[
	v = u - K\rho
	\] on $\Omega$ for some constant $K$ to be chosen later. Recalling that $u$ and $\rho$ are both negative on $\Omega$ and equal on $\Gamma^+$, we claim that for large enough $K$, $v$ cannot have a local minimum inside $\Omega$. Indeed, if $v$ has a minimum at $x_0 \in \Omega$, then at $x_0$, \[
	0 \leq u^{ij}v_{ij} = u^{ij}u_{ij} - Ku^{ij}\rho_{ij} \leq n - K\lambda_\text{min} \sum u^{ii} \stackrel{\eqref{C0BoundOnTrace}}{\leq} n - K\lambda_\text{min} \frac{n}{\|\psi\|_\infty} < 0 \, ,
	\] provided we choose \[
	K > \frac{\|\psi\|_\infty}{\lambda_\text{min}} \, .
	\] Here $\lambda_\text{min} > 0$ is the minimal eigenvalue of $D^2 \rho$, which is strictly positive by strict convexity, and we have employed equation \eqref{C0BoundOnTrace} to bound the trace of $u^{ij}$. If additionally we impose \[
	K > \frac{C_0}{\min_{\Gamma^-}|\rho|} \, ,
	\] then $v > 0$ on $\Gamma^-$. Thus the minimum occurs on $\Gamma^+$. Therefore we have $v_{\normal} \geq 0$, where $\normal$ is the inward-pointing unit normal field on $\Gamma^+$, and so \[
	|D u| = -u_\normal \leq -K\rho_\normal = K|D \rho| \, ,
	\] from which the bound \eqref{GradientBound} follows.
\end{proof}

\begin{remark}\label{Constants}
	Since we did not use the convexity of $\Gamma^-$ in the proof of Proposition \ref{Gradient}, the estimate \eqref{GradientBound} is also true for domains with non-convex inner boundary. Since $C_0$ does not depend on $\gamma_0$, it follows from Proposition \ref{Gradient} that for fixed $\Omega, \psi, \phi$, we have $C^1$ does not depend on $\gamma_0$ either.
\end{remark}


\subsection{The Double Tangential Derivative Estimate on $\Gamma^-$}\label{sec-tan}

In this section we prove a boundary $C^2$ estimate for the second order tangential derivatives on $\Gamma^-$ using a generalised barrier.

\begin{proposition}\label{th-C2}
	Suppose $u \in C^4(\Omega)\cap C^3(\overline{\Omega})$ is a strictly convex solution of \eqref{TheEquation}. Then there exists a constant $C_2 = C_2(n,\Omega,\Gamma^\pm,\psi,\phi)$ such that \[
	\sup_{\Gamma^-}|D_{\xi\xi} u| \leq C_2
	\] for any direction $\xi$ tangential to $\Gamma^-$ provided that \[
	2\max_{\substack{i = 1,...,n-1 \\ \bar{x} \in\Gamma^-}} \kappa_i(\bar{x}) < \gamma_0 + \max \left\{ 0 , \min_{\Gamma^-}\frac{\gamma_0 u + \phi}{M-u} \right\} \, ,
	\] where $\kappa_i(\bar{x})$ is the $i^{\text{th}}$ principal curvature of $\Gamma^-$ at $\bar{x}$, and $M = M(n,\Omega,\Gamma^\pm,\psi,\phi)$ is a constant.
\end{proposition}

It is convenient to denote the extension of the Neumann condition into $\Omega$ by $\chi(x, u)=\gamma_0 u + \phi(x)$, where by an abuse of notation $\phi$ is a smooth extension of its counterpart on the boundary. Let us consider a generalised auxiliary function modified from that of Lions-Trudinger-Urbas, \[
w = g(u)u_{\xi\xi} + a_k u_k + b + M|x|^2 \, ,
\] where $\xi$ is an arbitrary direction, $g(u)$ is a nonnegative function to be chosen, $M$ is a constant also to be chosen, and $a_k$ and $b$ are as in \cite{LTU}:
\begin{gather*}
a_k = 2(\xi\cdot \nu)(\chi_u\xi_k'-\xi_k'D_i\nu_k) \, , \\
b = 2(\xi\cdot \nu)\xi_k'\chi_{x_k} \, ,
\end{gather*}
$\xi'=\xi-(\xi\cdot\nu)\nu$ with $\nu$ a $C^{2,1}(\Omega)$ extension of the inner unit normal field on $\Gamma^-$.

\begin{lemma}\label{lem-G}
	Let $g(u) \geq 0$ for $u\in [-C_0, 0]$. The auxiliary function $w$ cannot have a local maximum inside $\Omega$ provided that the following three conditions hold:
	\begin{itemize}
		\item[(i)] \[
		g'' - \frac{2(g')^2}{g} \geq 0 \, , \quad u \in [-C_0,0] \, ,
		\]
		\item[(ii)] \[
		g' \geq 0 \, , \quad u \in [-C_0, 0] \, ,
		\]
		\item[(iii)] \[
		-\frac{g'}{g} \text{ is bounded on } [-C_0, 0] \, .
		\]
	\end{itemize}
\end{lemma}

\begin{proof}
	Suppose for contradiction that $w$ does have a local maximum at $x_0 \in \Omega$. Then at $x_0$ we have
	\begin{align}
	& w_i = g' u_i u_{\xi\xi} + g u_{i\xi\xi} + a_{k,i}u_k + a_ku_{ki} + b_i + 2Mx_i = 0 \, , \label{C2 wi} \\
	& w_{ij} = g'' u_i u_j u_{\xi\xi} + g' u_{ij}u_{\xi\xi} + g' u_i u_{j\xi\xi} + g' u_j u_{i\xi\xi} + g u_{ij\xi\xi} \label{C2 wij} \\
	& \qquad \ \ + a_{k,ij}u_k + a_{k,i}u_{kj} + a_{k,j}u_{ki} + a_k u_{kij} + b_{ij} + 2M\delta_{ij} \leq 0 \, , \nonumber
	\end{align}
	where $g' = g'(u)$. For clarity, we use a comma in the subscript to denote the derivatives of the components $a_k$. Now set $F^{ij} = u^{ij}$ so that
	\begin{equation*}
	F^{ij}u_{jk} = \delta^i_k \text{ and } F^{ij}u_{ij} = n \, ,
	\end{equation*}
	as in \cite{LTU}. \\
	
	We want to multiply equation \eqref{C2 wij} by $F^{ij}$ and eliminate the third order terms using \eqref{C2 wi}. First we show that the fourth order terms can be cancelled. Indeed if $F^{ij}$ and $u_{ij}$ are as above, then for any direction $\xi$ we have
	\begin{equation}\label{Uijxx}
	F^{ij}u_{ij\xi\xi} = \ln(\psi^n)_{\xi\xi} + F^{ij}u_{jk\xi}F^{kl}u_{li\xi} \, .
	\end{equation}
	To derive \eqref{Uijxx}, we first find an expression for $F^{ij}_\xi$. By the product rule,
	\begin{align*}
	F^{jk}u_{kl} = \delta^j_l & \implies F^{jk}u_{kl\xi} + F^{jk}_\xi u_{kl} = 0 \\
	& \implies F^{jk} u_{kl\xi}F^{il} + F^{jk}_\xi \delta^i_k = 0 \\
	& \implies F^{ij}_\xi = -F^{il}F^{jk}u_{kl\xi} \, .
	\end{align*}
	Taking the logarithm of both sides of \eqref{TheEquation} and differentiating in $\xi$ (or applying Jacobi's formula for the derivative of the determinant) we also have
	\begin{equation}\label{Why}
	F^{ij}u_{ij\xi} = \ln(\psi^n)_\xi \, .
	\end{equation}
	Differentiating \eqref{Why} in $\xi$ and substituting in the expression for $F^{ij}_\xi$ gives \[
	F^{ij}u_{ij\xi\xi} = \ln(\psi^n)_{\xi\xi} - F^{ij}_\xi u_{ij\xi} = \ln(\psi^n)_{\xi\xi} + F^{il}F^{jk}u_{kl\xi}u_{ij\xi} \, ,
	\] which is \eqref{Uijxx} after changing some of the indices.
	
	\begin{lemma}\label{Fijwij}
		From $b$, define $\overline{b_k} = 2(\xi\cdot \nu)\xi_k'$. With the above notation, we have
		\begin{gather*}
		F^{ij}w_{ij} = \left[ ng' + \left( g'' - \frac{2(g')^2}{g} \right)F^{ij}u_i u_j \right]u_{\xi\xi} - \frac{2g'}{g}F^{ij}u_i \left( u_k a_{k,j} + b_j + 2Mx_j \right) - \frac{2g'}{g} a_k u_k \\
		+ g\left( \ln(\psi^n)_{\xi\xi} + F^{ij}u_{jk\xi}F^{kl}u_{li\xi} \right) + F^{ij}\Bigl[ a_{k,ij}u_k + \overline{b_k}_{,ij}(\gamma_0 u_k + \phi_k) + 2\overline{b_k}_{,i} \phi_{kj} + \overline{b_k}\phi_{ijk} \Bigr] \\
		+ (a_k + \gamma_0 \overline{b_k}) \ln(\psi^n)_k + 2\sum (a_{k,k} + \gamma_0 \overline{b_k}_{,k}) + 2M\sum F^{ii} \, .
		\end{gather*}
	\end{lemma}	
	
	\begin{proof}
		Multiplying \eqref{C2 wij} by $F^{ij}$ and substituting in the expressions for $u_{i\xi\xi}$ from \eqref{C2 wi}, $F^{ij}u_{kij}$ from \eqref{Why} and $F^{ij}u_{ij\xi\xi}$ from \eqref{Uijxx} yields
		\begin{align*}
		F^{ij}w_{ij} = & \left( g'' F^{ij}u_i u_j + g' F^{ij}u_{ij} \right)u_{\xi\xi} - g' F^{ij}u_i \left[ \frac{g'}{g}u_j u_{\xi\xi} + \frac{1}{g}(a_{k,j} u_k + a_k u_{kj} + b_j + 2Mx_j) \right] \\
		& - g' F^{ij}u_j \left[ \frac{g'}{g}u_i u_{\xi\xi} + \frac{1}{g}(a_{k,i} u_k + a_k u_{ki} + b_i + 2Mx_i) \right] + g\left( \ln(\psi^n)_{\xi\xi} + F^{ij}u_{jk\xi}F^{kl}u_{li\xi} \right) \\
		& + a_{k,ij}F^{ij}u_k + a_{k,i}F^{ij}u_{jk} + a_{k,j}F^{ij}u_{ik} + a_k F^{ij}u_{kij} + F^{ij}b_{ij} + 2M\sum F^{ii} \\
		= & \left( g'' F^{ij}u_i u_j + ng' - \frac{2(g')^2}{g}F^{ij}u_i u_j \right)u_{\xi\xi} \\
		& - \frac{g'}{g}F^{ij}\bigl( u_i a_{k,j}u_k + u_i a_k u_{kj} + u_i b_j + 2Mu_i x_j + u_j a_{k,i}u_k + u_j a_k u_{ki} + u_j b_i + 2Mu_j x_i \bigr) \\
		& + g\left( \ln(\psi^n)_{\xi\xi} + F^{ij}u_{jk\xi}F^{kl}u_{li\xi} \right) + F^{ij}a_{k,ij}u_k + a_k \ln(\psi^n)_k + 2\sum a_{k,k} + F^{ij}b_{ij} + 2M\sum F^{ii} \, ,
		\end{align*}
		from which the result follows by index manipulations in the middle line, and using \eqref{Why} to obtain
		\begin{gather*}
		F^{ij}b_{ij} = F^{ij}\Bigl[ \overline{b_k}_{,ij}\chi_{x_k} + 2\overline{b_k}_{,i} \chi_{x_kx_j} + \overline{b_k}(\gamma_0u_{ijk} + \phi_{ijk}) \Bigr] \\
		= F^{ij}\Bigl[ \overline{b_k}_{,ij}(\gamma_0u_k + \phi_k) + 2\overline{b_k}_{,i} \phi_{jk} + \overline{b_k}\phi_{ijk} \Bigr] + 2\gamma_0 \overline{b_k}_{,k} + \gamma_0 \overline{b_k} \ln(\psi^n)_k \, . \qedhere
		\end{gather*}
	\end{proof}
	
	If we choose $g(u) = 1/(M-u)$, $M > 0$, then \[
	g(u) = \frac{1}{M - u} \, , \quad
	g'=\frac1{(M-u)^2} \, , \quad g'' - \frac{2(g')^2}{g} = 0 \, .
	\]
	
	Notice that $a_k$, $\overline{b_k}$ and their derivatives are bounded by some constant depending only on $\Omega$ and $\chi_u=\gamma_0$, and $|Du| \leq C_1$ with $C_1$ independent of $\gamma_0$ by Proposition \ref{Gradient}, see Remark \ref{Constants}. Also, $F^{ij}b_j = F^{ij}\overline{b_k}_{,j}\chi_{x_k} + F^{ij}\overline{b_k} \phi_{kj} + \gamma_0 \overline{b_i}$, which is bounded. With this choice of $g$, and using the fact $u \leq 0$ is convex, Lemma \ref{Fijwij} implies for large enough $M$
	\begin{gather*}
	F^{ij}w_{ij} \geq \frac{n u_{\xi\xi}}{(M-u)^2} - \frac{2}{M-u}F^{ij}u_i \left( u_k a_{k, j} + b_j + 2Mx_j \right) - \frac{2}{M-u} a_k u_k + \frac{\ln(\psi^n)_{\xi\xi}}{M-u} \\
	+ F^{ij}\Bigl[ a_{k,ij}u_k + \overline{b_k}_{,ij}(\gamma_0 u_k + \phi_k) + 2\overline{b_k}_{,i} \phi_{kj} + \overline{b_k}\phi_{ijk} \Bigr] + (a_k + \gamma_0 \overline{b_k}) \ln(\psi^n)_k + 2\sum (a_{k,k} + \gamma_0 \overline{b_k}_{,k}) + 2M\mathcal{T} \, \\
	\geq
	- \frac{2}{M-u}F^{ij}u_i \left( u_k a_{k, j} + b_j + 2Mx_j \right) + F^{ij}\Bigl[ a_{k,ij}u_k + \overline{b_k}_{,ij}(\gamma_0 u_k + \phi_k) + 2\overline{b_k}_{,i} \phi_{kj} + \overline{b_k}\phi_{ijk} \Bigr] \\
	- \frac{2}{M-u} a_k u_k + \frac{\ln(\psi^n)_{\xi\xi}}{M-u} + (a_k + \gamma_0 \overline{b_k}) \ln(\psi^n)_k + 2\sum (a_{k,k} + \gamma_0 \overline{b_k}_{,k}) + 2M\mathcal{T} \,
	\\
	\geq
	- \frac{2}{M}C_1 \left[ C_1 \|D a\|_\infty + (C_1 + \|D\phi\|_\infty) \|D \overline{b_k}\|_\infty + \|\overline{b_k}\|_\infty \|D^2\phi\|_\infty + 2M\sup_{\Gamma^+} |x| \right] \mathcal{T} \\
	- \Bigl[ C_1 \|D^2 a\|_\infty + (C_1 + \|D\phi\|_\infty) \|D^2 \overline{b_k}\|_\infty + 2\|D\overline{b_k}\|_\infty \|D^2\phi\|_\infty + \|\overline{b_k}\|_\infty \|D^3 \phi\|_\infty \Bigr] \mathcal{T} - \gamma_0\|\overline{b_k}\|_\infty \\
	- \frac{2}{M} \|a\|_\infty C_1 - \frac{n}{M} \|D^2 \ln(\psi)\|_\infty - n \left( \|a\|_\infty + \gamma_0 \|\overline{b_k}\|_\infty \right) \|D \ln(\psi)\|_\infty - 2n(\|Da\|_\infty + \gamma_0 \|\overline{b_k}\|_\infty) + 2M\mathcal{T} \,
	\\
	\geq
	\left\{
	2M
	- \frac{2}{M}C_1 \Bigl[ C_1 \|D a\|_\infty + (C_1 + \|D\phi\|_\infty) \|D \overline{b_k}\|_\infty + \|\overline{b_k}\|_\infty \|D^2\phi\|_\infty + 2M\sup_{\Gamma^+} |x| \Bigr] \right. \\
	\biggl. - \Bigl[ \|D ^2a\|_\infty C_1 + (C_1 + \|D\phi\|_\infty) \|D^2 \overline{b_k}\|_\infty + 2\|D\overline{b_k}\|_\infty \|D^2\phi\|_\infty + \|\overline{b_k}\|_\infty \|D^3 \phi\|_\infty \Bigr] \biggr\} \mathcal{T}
	\\
	- \gamma_0\|\overline{b_k}\|_\infty - \frac{n}{M} \|D^2 \ln(\psi)\|_\infty - n \left( \|a\|_\infty + \gamma_0 \|\overline{b_k}\|_\infty \right) \|D \ln(\psi)\|_\infty - 2n(\|Da\|_\infty + \gamma_0 \|\overline{b_k}\|_\infty) - \frac{2}{M} \|a\|_\infty C_1
	\\
	= \left\{ 2M - \frac{\ell_1}{M} - \ell_2 \right\} \frac{n}{\|\psi\|_{\infty}} - \frac{n}{M} \|D^2 \ln(\psi)\|_\infty - \frac{2}{M} \|a\|_\infty C_1 - \ell_3 \, ,
	\end{gather*}
provided the coefficient of $\mathcal{T}$ is positive, and where
	\begin{gather*}
	\ell_1= 2C_1 \Bigl( C_1 \|D a\|_\infty + (C_1 + \|D\phi\|_\infty) \|D \overline{b_k}\|_\infty + \|\overline{b_k}\|_\infty \|D^2\phi\|_\infty  \Bigr) \, , \\
	\ell_2=\|D ^2a\|_\infty C_1 + (C_1 + \|D\phi\|_\infty) \|D^2 \overline{b_k}\|_\infty + 2\|D\overline{b_k}\|_\infty \|D^2\phi\|_\infty + \|\overline{b_k}\|_\infty \|D^3 \phi\|_\infty + 4C_1\sup_{\Gamma^+} |x| \, , \\
	\ell_3 = \gamma_0\|\overline{b_k}\|_\infty + n \left( \|a\|_\infty + \gamma_0 \|\overline{b_k}\|_\infty \right) \|D \ln(\psi)\|_\infty + 2n(\|Da\|_\infty + \gamma_0 \|\overline{b_k}\|_\infty) \, .
	\end{gather*}
	
	In particular, choosing
	\begin{gather}\label{MBound}
	M > \max\left\{ \frac{\ell_2 + \sqrt{\ell_2^2 + 4 \ell_1}}{2} \, , \, \frac{\|\psi\|_\infty}{n} \left( \frac{\ell_3}{2} + \sqrt{\left( \frac{\ell_3}{2} \right)^2 + \frac {n\left( n \|D^2 \ln(\psi)\|_\infty + 2\|a\|_\infty C_1 \right)}{\|\psi\|_\infty}} \right) \right\} \, ,
	\end{gather}
	we obtain 
	\begin{gather*}
	F^{ij}w_{ij} \geq \frac{Mn}{\|\psi\|_\infty} - \frac{n}{M} \|D^2 \ln(\psi)\|_\infty - \frac{2}{M} \|a\|_\infty C_1 - \ell_3 > 0 \, .
	\end{gather*}
	
	This contradicts the assumption of $w$ having a maximum at $x_0$. Thus the maximum of $w$ must occur on $\d\Omega$. It can readily be seen that this reasoning carries over to any choice of function $g$ which has the properties from Lemma \ref{lem-G}.
\end{proof}

\subsubsection{Proof of Proposition \ref{th-C2}}

We now show that if the maximum of $w$ occurs on $\Gamma^-$ then the second order tangential derivatives of $u$ are bounded. Suppose that the maximum occurs at $x_0 \in \Gamma^-$. Let $\xi$ be a unit tangent vector to $\Gamma^-$ at $x_0$, and let $\pi$ be the 2-plane spanned by $\xi$ and $\nu$, the unit normal to $\Gamma^-$. The intersection of $\pi$ with $\Gamma^-$ is a curve denoted by $\sigma(s)$, passing through $x_0$ with tangent vector $\xi$ at $x_0$, where $s$ is an arc-length parameter. From the Neumann boundary condition, we have
\begin{equation}\label{NeumannRewritten}
u_\nu (\sigma(s)) = \chi(\sigma(s), u(\sigma(s))) \, .
\end{equation}
Differentiating \eqref{NeumannRewritten} twice in $s$ gives
\begin{gather*}
u_{\dot{\nu}} + u_{\nu\dot{\sigma}} = \chi_u u_{\dot{\sigma}} + D_x \chi\cdot \dot{\sigma}(s) \, , \\
u_{\ddot{\nu}} + 2u_{\dot{\nu}\dot{\sigma}} + u_{\nu\dot{\sigma}\dot{\sigma}} + u_{\nu\ddot{\sigma}} = u_{\dot{\sigma}} (D_x \chi_u \cdot \dot{\sigma}) + \chi_{uu}(u_{\dot{\sigma}})^2 + \chi_u u_{\ddot{\sigma}} + \chi_u u_{\dot{\sigma}\dot{\sigma}} \\
\qquad \qquad \qquad \qquad \qquad \qquad \qquad + \dot{\sigma}\cdot D^2_{xx}\chi\dot{\sigma} + (D_x \chi_{u}\cdot\dot{\sigma})u_{\dot{\sigma}} + D_x \chi\cdot\ddot{\sigma} \, . \nonumber
\end{gather*}
Rewriting these using the definition of the Frenet frame, we have
\begin{gather}
\kappa_\xi u_{\xi} + u_{\nu\xi} = \chi_u u_{\xi} + D_x \chi\cdot\xi \, , \nonumber \\
u_{\ddot{\nu}} + 2\kappa_\xi u_{\xi\xi} + u_{\nu\xi\xi} - \kappa_\xi u_{\nu\nu} = u_\xi D_x \chi_u \cdot \xi + \chi_{uu}(u_\xi)^2 - \kappa_\xi \chi_u u_\nu + \chi_u u_{\xi\xi} \label{FrenetU} \\
\qquad \qquad \qquad \qquad \qquad \qquad \qquad + \xi\cdot D^2_{xx}\chi\xi + (D_x \chi_{u}\cdot\xi)u_{\xi} - \kappa_\xi D_x \chi\cdot\nu \, , \nonumber
\end{gather}
where $\kappa_\xi > 0$ is the normal curvature of $\Gamma^-$ in the $\xi$ direction. From \eqref{FrenetU},
\begin{equation}\label{unuxixi}
u_{\nu\xi\xi} = \kappa_\xi u_{\nu\nu} + (\chi_u - 2\kappa_\xi) u_{\xi\xi} + O(1) \, .
\end{equation}
On the other hand, since $w$ has a maximum at $x_0$, we have
\begin{gather}
0 \geq {} \d_\nu w(x_0) \stackrel{\eqref{C2 wi}}{=} \frac{u_\nu u_{\xi\xi}}{(M-u)^2} + \frac{u_{\nu\xi\xi}}{M-u} + a_{k,\nu}u_k + a_ku_{k\nu} + b_\nu + 2M\inprod{x}{\nu} \nonumber \\
\implies u_{\nu\xi\xi} \leq -\frac{u_\nu u_{\xi\xi}}{M-u} - (M-u)(a_{k,\nu}u_k + a_ku_{k\nu} + b_\nu + 2M\inprod{x}{\nu}) \, . \label{unuxixi<}
\end{gather}
Combining \eqref{unuxixi} and \eqref{unuxixi<} and noting that $a_k = 0$ and $b_\nu = O(1)$ on $\d\Omega$ if $\xi\perp\nu$, we obtain \[
\left( \chi_u - 2\kappa_\xi + \frac{\chi(x_0)}{M-u} \right)u_{\xi\xi} + \kappa_\xi u_{\nu\nu} \leq O(1) \, .
\] This works as an estimate provided the coefficient of $u_{\xi\xi}$ is positive. If $\chi$ is negative somewhere on $\Gamma^-$, then this is true for any $2\kappa_\xi < \gamma_0$ provided that we choose $M$ large enough that \[
2\kappa_\xi < \gamma_0 + \min_{\Gamma^-}\frac{\gamma_0 u + \phi}{M - u} \leq \frac{\gamma_0 M + \phi(x_0)}{M - u} < \gamma_0
\] holds in addition to the condition \eqref{MBound} already imposed. On the other hand, if $\chi \geq 0$ on $\Gamma^-$ then we need only require \eqref{MBound} and \[
2\kappa_\xi < \gamma_0 + \min_{\Gamma^-}\frac{\gamma_0 u + \phi}{M-u} \, .
\] Hence $u_{\xi\xi}$ is bounded if \[
2\kappa_\xi < \gamma_0 + \max \left\{ 0 , \min_{\Gamma^-}\frac{\gamma_0 u + \phi}{M-u} \right\} \, .
\] This concludes the proof of Proposition \ref{th-C2}. \qed

\begin{remark}
	Let $g$ be as in Lemma \ref{lem-G} and suppose $\gamma_0 = 0$. Then proceeding as above,
	\begin{equation*}
	0 \geq \d_\nu w(x_0) = g' u_\nu u_{\xi\xi} + g u_{\nu\xi\xi} + a_{k\nu}u_k + a_ku_{k\nu} + b_\nu + 2M\inprod{x}{\nu} \, ,
	\end{equation*}
	and thus
	\begin{gather*}
	\kappa_\xi u_{\nu\nu} + (\chi_u - 2\kappa_\xi) u_{\xi\xi} + O(1) = u_{\nu\xi\xi} \leq -\frac{g'}{g}{u_\nu u_{\xi\xi}} - \frac{1}{g}(a_{k\nu}u_k + a_ku_{k\nu} + b_\nu + 2M\inprod{x}{\nu}) \\
	\implies \kappa_\xi u_{\nu\nu} + \left(\chi_u - 2\kappa_\xi + \frac{g'}{g}\chi\right) u_{\xi\xi} \leq - \frac{1}{g}(a_{k\nu}u_k + a_ku_{k\nu} + b_\nu + 2M\inprod{x}{\nu}) + O(1+\gamma_0) \, .
	\end{gather*}
	In other words, if there is $g$ satisfying (i)-(iii) in Lemma \ref{lem-G} and \[
	2\kappa_\xi < \frac{\left( \chi g \right)_u}{g}  \, , 
	\] then one has the desired estimate for $u_{\xi\xi}$ when $\gamma_0 = 0$. 
\end{remark}

\begin{remark}
	If the maximum of $w$ occurs on $\Gamma^+$ then we can estimate the second-order derivatives in the standard way, see \cite{CNS-1} or Chapter 17 of \cite{GilbargTrudinger}.
\end{remark}


\subsection{Explicit Solutions and a Counterexample}\label{sec-explicit}

In this section we investigate some explicit solutions of \eqref{TheEquation} and demonstrate a few of their properties. In the first part we derive radial solutions of the equation in a setting where $\Gamma^\pm$ are concentric hyperspheres, and use these to construct a sequence of solutions for which the second normal derivative grows without bound. These solutions are small perturbations of the function $f(r), r=|x|$, where $f$ solves $f'(r) = \sqrt{r^2 - R_-^2}$ for $r > R_- > 0$, and $f'(r)=0$ for $0 \leq r < R_-$, where $R_-$ is a constant.

In the second part we consider non-radially symmetric domains and investigate some possible behaviours of $u_\nu$ and $\phi$. In particular we demonstrate that the normal derivative on the inner boundary may be zero or even negative at some points on $\Gamma^-$, justifying our caution over the sign of $\chi$ at the end of the proof of Proposition \ref{th-C2}. Throughout this section we assume $\psi(x) \equiv \psi$ to be a constant function. Except where stated, we will also work in $n=2$ dimensions.

\subsubsection{Radial Solutions in the Concentric Setting}

Consider the 2-dimensional Monge-Amp\`ere equation with $\psi(x) \equiv \psi$ on the domain $\Omega = B_{R_+}\setminus\overline{B_{R_-}}$, where $R_+ > R_-$.

\begin{proposition}
	In $2$ dimensions, if $u \in C^2(\overline{\Omega})$ is a solution of $\det[D^2 u] = \psi^2$ on a radially symmetric planar domain $\Omega$, then in polar coordinates $(r,\theta)$, $u$ satisfies \[
	\frac{u_{rr}u_r}{r} + \frac{u_{rr}u_{\theta\theta}}{r^2} - \frac{1}{r^2} \left( u_{r\theta} - \frac{u_\theta}{r} \right)^2 = \psi^2 \, .
	\] In particular, if $u = u(r)$ is radial, then $u$ satisfies
	\begin{equation}\label{RadialMongeAmpere}
	\frac{u_{rr}u_r}{r} = \psi^2 \, .
	\end{equation}
\end{proposition}

\begin{proof}
	We perform a standard change of variables to derive the Hessian matrix in polar coordinates, and then take its determinant. Setting $x = r\cos\theta$, $y = r\sin\theta$,
	\begin{gather*}
	\frac{\d^2 u}{\d x^2} = u_{rr}\cos^2\theta + u_r \frac{\sin^2\theta}{r} - u_{r\theta}\frac{2\sin\theta\cos\theta}{r} + u_{\theta\theta}\frac{\sin^2\theta}{r^2} + u_\theta \frac{2\sin\theta\cos\theta}{r^2} \, , \\
	\frac{\d^2 u}{\d y^2} = u_{rr}\sin^2\theta + u_r \frac{\cos^2\theta}{r} + u_{r\theta}\frac{2\sin\theta\cos\theta}{r} + u_{\theta\theta}\frac{\cos^2\theta}{r^2} - u_\theta \frac{2\sin\theta\cos\theta}{r^2} \, , \\
	\frac{\d^2 u}{\d x \d y} = u_{rr}\sin\theta\cos\theta - u_r \frac{\sin\theta\cos\theta}{r} + u_{r\theta}\frac{\cos^2\theta - \sin^2\theta}{r} - u_{\theta\theta}\frac{\sin\theta\cos\theta}{r^2} + u_\theta \frac{\sin^2\theta - \cos^2\theta}{r^2} \, .
	\end{gather*}
	The determinant is now a simple but lengthy calculation from these expressions.
\end{proof}

From now on we will focus purely on solving \eqref{RadialMongeAmpere} for various values of $u_r(R_-)$, the normal derivative on the inner boundary $\Gamma^-$. By defining $w(r) = u_r$ and denoting differentiation in $r$ by an apostrophe, \eqref{RadialMongeAmpere} becomes the separable equation
\begin{gather*}
\frac{w'w}{r} = \psi^2 \implies w(r)^2 - w(R_-)^2 = \psi^2(r^2 - R_-^2) \\
\implies u'(r) = w(r) = \psi\sqrt{r^2 - R_-^2 + \frac{u'(R_-)^2}{\psi^2}} \, .
\end{gather*}
Here, we temporarily ignore the Neumann boundary condition on $\Gamma^-$ from $\eqref{TheEquation}$ and simply take $u_\nu = u'(R_-)$ to be some constant determining the slope of the solution at $\Gamma^-$. We will eventually rectify this by making a particular choice of the function $\phi$.

\begin{proposition}
	The solutions of \eqref{RadialMongeAmpere} for a particular choice of $u'(R_-)$ are given by \[
	u(r) = \frac{\psi}{2} r \sqrt{r^2 - R_-^2 + \frac{u'(R_-)^2}{\psi^2}} - \frac{\psi}{2} \left( R_-^2 - \frac{u'(R_-)^2}{\psi^2} \right) \ln \left( r + \sqrt{r^2 - R_-^2 + \frac{u'(R_-)^2}{\psi^2} } \right) + C \, ,
	\] where $C$ is the constant chosen so that $u(R_+) = 0$, namely \[
	C = -\frac{\psi}{2} R_+ \sqrt{R_+^2 - R_-^2 + \frac{u'(R_-)^2}{\psi^2}} + \frac{\psi}{2} \left( R_-^2 - \frac{u'(R_-)^2}{\psi^2} \right) \ln \left( R_+ + \sqrt{R_+^2 - R_-^2 + \frac{u'(R_-)^2}{\psi^2} } \right) \, .
	\]
\end{proposition}

\begin{proof}
	We calculate \[
	u_r = \psi \sqrt{r^2 - R_-^2 + \frac{u'(R_-)^2}{\psi^2}} \, , \quad u_{rr} = \frac{\psi r}{\sqrt{r^2 - R_-^2 + \frac{u'(R_-)^2}{\psi^2}}} \, ,
	\] and the result follows directly by substitution into \eqref{RadialMongeAmpere}.
\end{proof}

\begin{remark}
	When $u'(R_-) = \psi R_-$, $u$ reduces to the quadratic solution \[
	u(r) = \frac{\psi}{2}r^2 - \frac{\psi}{2}R_+^2 \, .
	\]
\end{remark}

\subsubsection{A sequence of solutions for which the second normal derivative grows to $\infty$}

Let $d_k$ be a sequence of positive real numbers tending to 0. Define
\begin{align*}
\phi_k \eqdef {} & d_k - \frac{\gamma_0}{2} \left[ R_- d_k - \psi\left( R_-^2 - \frac{d_k^2}{\psi^2} \right) \ln  \left( R_- + \frac{d_k}{\psi} \right) \right] \\
& + \frac{\psi\gamma_0}{2} \left[ R_+ \sqrt{R_+^2 - R_-^2 + \frac{d_k^2}{\psi^2}} - \left( R_-^2 - \frac{d_k^2}{\psi^2} \right) \ln \left( R_+ + \sqrt{R_+^2 - R_-^2 + \frac{d_k^2}{\psi^2} } \right) \right] \, .
\end{align*}
This choice of $\phi$ is continuous and increasing in $d_k$ (and thus decreasing in $k$). If $R_+ = \mu R_-$, $\mu\in(1,\infty)$, then
\begin{equation}\label{phiklowerbound}
\lim_{k\rightarrow\infty} \phi_k = \phi^\psi_\infty \eqdef \frac{\psi\gamma_0}{2} \left[ \mu\sqrt{\mu^2 - 1} - \ln(\mu + \sqrt{\mu^2 - 1}) \right] R_-^2 \, .
\end{equation}
Note that instead of defining $\phi_k$ in terms of the sequence $d_k$, we could equivalently pick a decreasing sequence of real numbers $\phi_k$ converging to the positive lower bound depending on $\psi,\gamma_0,R_+$ and $R_-$ given in \eqref{phiklowerbound} and determine the sequence $d_k$ from the formula above.

\begin{remark}
	It will be important later on to note the dependence of $\phi^\psi_\infty$ on $\psi$, in particular that $\phi^\psi_\infty$ is increasing in $\psi$.
\end{remark}

\begin{proposition}\label{Soln}
	The strictly convex radial solution $u^{(k)}$ of
	\begin{equation}\label{RadialEquation}
	\begin{cases}
	\det[D^2 u^{(k)}] = \psi^2  & \text{ in } \Omega = B_{R_+}\setminus \overline{B_{R_-}} \, , \\
	u^{(k)} = 0  & \text{ on } \Gamma^+ = \d B_{R_+} \, , \\
	u^{(k)}_\nu = \gamma_0 u^{(k)} + \phi_k & \text{ on } \Gamma^- = \d B_{R_-} \, ,
	\end{cases}
	\end{equation}
	is given by
	\begin{gather}
	u^{(k)} (x) = \frac{\psi}{2} \left[ |x| \sqrt{|x|^2 - R_-^2 + \frac{d_k^2}{\psi^2}} - \left( R_-^2 - \frac{d_k^2}{\psi^2} \right) \ln \left( |x| + \sqrt{|x|^2 - R_-^2 + \frac{d_k^2}{\psi^2} } \right) \right. \qquad \qquad \label{u(k)(x)} \\
	\qquad \qquad \left. - R_+ \sqrt{R_+^2 - R_-^2 + \frac{d_k^2}{\psi^2}} + \left( R_-^2 - \frac{d_k^2}{\psi^2} \right) \ln \left( R_+ + \sqrt{R_+^2 - R_-^2 + \frac{d_k^2}{\psi^2} } \right) \right] \, . \nonumber
	\end{gather}
\end{proposition}

\begin{proof}
	First, we check the boundary conditions. By construction $u^{(k)}$ satisfies the Dirichlet condition on $\Gamma^+$. To check the Neumann condition, we define $D_k \eqdef -R_-^2 + d_k^2/\psi^2$ and differentiate to find
	\begin{gather*}
	D u^{(k)} (x) = \psi \sqrt{|x|^2 + D_k} \frac{x}{|x|} \, , \\
	u^{(k)}_\nu |_{|x|=R_-} = \frac{x}{R_-} \cdot \psi \sqrt{|x|^2 + D_k} \frac{x}{|x|} |_{|x|=R_-} = d_k \, .
	\end{gather*}
	Note that $|x|^2 + D_k > d_k^2 /\psi^2 > 0$ for all $k$ since $|x| > R_-$. We also have on $\Gamma^-$,
	\begin{align*}
	\gamma_0 u^{(k)}|_{|x|=R_-} = {} & \frac{\psi \gamma_0}{2} \left[ R_- \frac{d_k}{\psi} - \left( R_-^2 - \frac{d_k^2}{\psi^2} \right) \ln \left( R_- + \frac{d_k}{\psi} \right) \right. \\
	& \left. - R_+ \sqrt{R_+^2 - R_-^2 + \frac{d_k^2}{\psi^2}} + \left( R_-^2 - \frac{d_k^2}{\psi^2} \right) \ln \left( R_+ + \sqrt{R_+^2 - R_-^2 + \frac{d_k^2}{\psi^2} } \right) \right] = d_k - \phi_k \, ,
	\end{align*}
	and thus the Neumann condition holds.
	
	Second, we check that the Hessian equation itself is satisfied. We have \[
	D^2 u^{(k)} = \frac{\psi}{|x|^3\sqrt{|x|^2 + D_k}} \left( \begin{array}{cc}
	|x|^4 + D_k y^2 & - D_k xy \\
	- D_k xy & |x|^4 + D_k x^2
	\end{array} \right) \, ,
	\] and so $\det[D^2 u^{(k)}] = \psi^2$ as required. Note that moreover \[
	\tr[D^2 u^{(k)}] = \frac{\psi}{|x| \sqrt{|x|^2 + D_k}}  \left( 2|x|^2 + D_k \right) \geq \frac{\psi}{|x| \sqrt{|x|^2 + D_k}} \left( R_-^2 + \frac{d_k^2}{\psi^2} \right) > 0 \, ,
	\] so $D^2 u^{(k)}$ is positive definite on $\overline{\Omega}$ (meaning $u^{(k)}$ is strictly convex), and all the second derivatives are continuous on $\overline{\Omega}$ for $d_k > 0$.
\end{proof}

Proposition \ref{Soln} tells us that the $u^{(k)}$ are perfectly good solutions of \eqref{TheEquation} on this domain for these choices of $\phi$. However, we now calculate their second normal derivative on $\Gamma^-$. We find
\begin{gather*}
u_{\nu\nu} = \psi\frac{x}{R_-^2} \cdot D (|x|\sqrt{|x|^2 + D_k}) = \frac{\psi}{R_-^2} \frac{2|x|^2 + D_k}{\sqrt{|x|^2 + D_k}} |x| \\
\implies u_{\nu\nu} |_{|x|=R_-} = \frac{\psi^2 R_-}{d_k} + \frac{d_k}{R_-} \, .
\end{gather*}
This is unbounded as $d_k \longrightarrow 0$, showing that in general it is not possible to bound the second normal derivative of solutions on $\Gamma^-$.

\begin{remark}
	This example also shows that for $\phi < \phi^\psi_\infty$, \eqref{RadialEquation} has no strictly convex solution $u \in C^{2,\alpha}$.
\end{remark}

\smallskip \smallskip \smallskip \smallskip
A similar argument can be made for radial solutions in the $n$-dimensional setting.

\begin{proposition}
	If $u\in C^2(\overline{\Omega})$ is a radial solution of $\det[D^2 u] = \psi^n$ on an $n$-dimensional radially symmetric domain $\Omega$, then
	\begin{equation}\label{RadialnMongeAmpere}
	u_{rr} \left( \frac{u_r}{r} \right)^{n-1} = \psi^n \, .
	\end{equation}
\end{proposition}

\begin{proof}
	Suppose $u = u(r)$, $r=|x|$. Then
	\begin{gather*}
	u_i = u_r \frac{x_i}{r} \, , \\
	u_{ij} = \frac{u_r}{r} \left( \delta_{ij} - \frac{x_i x_j}{r^2} \right) + u_{rr} \frac{x_i x_j}{r^2} \, .
	\end{gather*}
	Sylvester's Determinant Theorem then gives the required result.
\end{proof}

\begin{proposition}
	There exists a sequence of strictly convex, radially symmetric solutions of \eqref{RadialnMongeAmpere} for which the second normal derivative grows without bound.
\end{proposition}

\begin{proof}
	We show the existence of such a sequence without constructing the functions $u^{(k)}$ explicitly as we did in the $n=2$ setting. Let $d_k$ be a sequence of positive real numbers converging to 0. Integrating \eqref{RadialnMongeAmpere} once and fixing the Neumann condition on $\Gamma^-$ such that $u^{(k)}_\nu(R_-) = u^{(k)}_r(R_-) = d_k$, we have \[
	u^{(k)}_r (r) = \sqrt[n]{\psi^n r^n - \psi^n R_-^n + d_k^n} \, .
	\] Differentiating once in $r$ (or substituting back into \eqref{RadialnMongeAmpere}) gives \[
	u^{(k)}_{\nu\nu} (r) = u^{(k)}_{rr} (r) = \frac{\psi r^{n-1}}{\left( r^n - R_-^n + (d_k/\psi)^n \right)^{1-\frac{1}{n}}} \, ,
	\] and thus \[
	u^{(k)}_{\nu\nu} (R_-) = \frac{\psi^n R_-^{n-1}}{d_k^{n-1}} \longrightarrow \infty \quad \text{ as } k \longrightarrow \infty \, . \qedhere
	\]
\end{proof}

\subsubsection{Solutions in the Skewed-Concentric Setting}

Let $R_+ > R_-$, and let $\gamma_-, \gamma_+$ be vectors satisfying $|\gamma_-| < R_-$, $\gamma_+ < R_+$ and $|\gamma_+ - \gamma_-| < R_+ - R_-$. Set
\begin{gather*}
\Gamma^+ = \d B_{R_+}(\gamma_+) \, , \quad \Gamma^- = \d B_{R_-}(\gamma_-) \, , \\
\Omega = B_{R_+}(\gamma_+) \setminus \overline{B_{R_-}(\gamma_-)} \, ,
\end{gather*}
so that $\gamma_-$ is the centre of the inner circle and $\gamma_+$ is the centre of the outer one. The conditions above ensure that both circles contain the origin and that $\Gamma^-$ is strictly contained within $\Gamma^+$.

Consider the quadratic solution of \eqref{TheEquation} on $\Omega$ given by \[
u(x) = \frac{\psi}{2} |x - \gamma_+|^2 - \frac{\psi}{2} R_+^2 \, .
\] For $x\in\Gamma^-$, $x = \gamma_- + R_- \nu$, where $\nu$ is the inward-pointing normal at the point $x$. Hence we have on $\Gamma^-$ \[
u_\nu = \nu\cdot \psi(x - \gamma_+) = \psi \left[ (\gamma_- - \gamma_+)\cdot\nu + R_- \right] \, .
\] We see from this expression that for certain combinations of $\gamma_-,\gamma_+$ and $R_-$, there will be points on $\Gamma^-$ where $u_\nu = 0$ or even $u_\nu < 0$. For instance, if \[
\gamma_- = \left( \frac{1}{4} , 0 \right) \, , \quad \gamma_+ = (1,0) \, , \quad R_- = \frac{1}{2} \, , \quad R_+ = 2 \, ,
\] then \[
u_\nu\left( \frac{3}{4} , 0 \right) = -\frac{\psi}{4} \, , \quad \text{ and } \quad u_\nu \left( \frac{7}{12} , \pm \frac{\sqrt{5}}{6} \right) = 0 \, .
\] Similarly, this $u$ is the solution for \[
\phi(x) = \psi \left( (\gamma_- - \gamma_+)\cdot\nu + R_- \right) + \frac{\psi\gamma_0}{2} \left( R_+^2 - |\gamma_- - \gamma_+ + R_- \nu|^2 \right) \, ,
\] so for small enough $\gamma_0,R_-$ and large enough $|\gamma_+ - \gamma_-|$ we will have points where $\phi$ is zero or negative. \\

There are also solutions in this skew-concentric setting corresponding to the non-quadratic solutions on the concentric domain. If we define
\begin{equation}\label{v(k)(x)}
v^{(k)}(x) = u^{(k)}(x-\gamma_+) \, ,
\end{equation}
with $u^{(k)}$ given by \eqref{u(k)(x)}, then $v^{(k)}$ is a solution of \eqref{TheEquation} provided $\Omega$ and $R_-$ are chosen such that $|x-\gamma_+| \geq R_-$ for any $x \in \Omega$. From these we can generate more examples demonstrating that the second normal derivative cannot be bounded in general.


\section{The Subsolution Condition and Existence for $\psi = \psi(x)$}\label{sec-ExtraConditions}

\subsection{The Double Normal Estimate in a Restricted Setting}\label{sec-normal}

In this section, we prove Theorem \ref{main-th-nor}, which states that under certain restrictions an estimate for $u_{\nu\nu}$ on $\Gamma^-$ can be obtained. We do this by estimating the mixed derivatives $u_{\xi\nu}$ from above and the tangential derivatives $u_{\xi\xi}$ from below on the inner boundary.

\begin{remark}
	On $\Gamma^-$, the  mixed tangential-normal derivatives are estimated directly from the Neumann condition:
	\begin{equation}\label{uxinu}
		|u_{\xi\nu}| = |\xi\cdot D (\gamma_0 u + \phi)| \leq \gamma_0 C_1 + \sup_{\Gamma^-} |D\phi| \eqdef C_3 \, .
	\end{equation}
\end{remark}

\subsubsection{Setup}\label{Setup}

Let $x_0 \in \Gamma^-$, and $\xi$ be a unit vector tangential to $\Gamma^-$ at $x_0$. Let $\gamma : (-a,a) \longrightarrow \Gamma^-$ be a unit-speed geodesic curve on $\Gamma^-$ passing through $x_0$ with the following properties:
\begin{enumerate}
	\item[(i)] $\gamma(0) = x_0$,
	\item[(ii)] $\gamma'(s)|_{s=0} = \xi$,
	\item[(iii)] $\gamma''(s)|_{s=0} = -\kappa_{\xi} \nu$.
\end{enumerate}
Here, $\kappa_\xi$ is the normal curvature of $\Gamma^-$ in the direction $\xi$. The existence of such a geodesic follows from standard ODE theory and the smoothness of $\Gamma^-$.

\begin{proposition}\label{Double Normal Estimate}
	Suppose $u\in C^4(\Omega)\cap C^3(\overline{\Omega})$ is a strictly convex solution of \eqref{TheEquation}, and that there exists a strictly convex subsolution $\underline{u} \in C^4(\Omega)\cap C^3(\overline{\Omega})$ satisfying
	\begin{equation*}
		\begin{cases}
			\det[D^2 \underline{u}] \geq \psi^n(x)  & \text{ in } \Omega \, , \\
			\underline{u} = 0  & \text{ on } \Gamma^+ \, , \\
			\underline{u}_\nu = \gamma_0 \underline{u} + \phi(x) & \text{ on } \Gamma^- \, ,
		\end{cases}
	\end{equation*}
	with the same $\phi$ as in \eqref{TheEquation}. For each $x_0\in\Gamma^-$ define the quantity $U(s) \eqdef u_\nu(\gamma(s)) - \underline{u}_\nu(\gamma(s))$ along the curve $\gamma(s)$ satisfying properties (i), (ii) and (iii). Suppose additionally that there exists $\tau$ (not depending on the choice of $\xi$, $\underline{u}$ or $x_0$) such that the differential inequality
	\begin{equation}
		\left( \frac{1}{\gamma_0} \frac{d^2}{ds^2} + \kappa_\xi \right) U(s) + \underline{u}_{\xi\xi} \geq \tau
	\end{equation}
	holds at $s=0$ for every $x_0\in\Gamma^-$. Then $u_{\nu\nu}$ is bounded with \[
	\sup_{\Gamma^-}|D_{\nu\nu} u| \leq C_4 \, .
	\]
\end{proposition}

\subsubsection{Motivation}
Consider the 2-dimensional concentric setting with $\Gamma^- = \d B_{1}$. We parametrise $\Gamma^-$ by $\gamma(s) = (\cos s, \sin s)$, $s\in(-\pi,\pi]$. This parametrisation has the advantage that the normal at each point of $\Gamma^-$ is $\nu(s) = \gamma(s)$, and the tangent vector at each point of $\Gamma^-$ is $\xi(s) = \gamma'(s)$.

Suppose $u\in C^3(\overline{\Omega})$ is a solution of \eqref{TheEquation}. Considering the restriction of $u$ on $\Gamma^-$ as a function of $s$, we have
\begin{gather*}
	\frac{d}{ds} u(\gamma(s)) = \gamma' \cdot D u = u_\xi \, , \\
	\frac{d^2}{ds^2} u(\gamma(s)) = \gamma' \cdot (D^2 u) \gamma' + \gamma'' \cdot D u = u_{\xi\xi} - \kappa_\xi u_\nu \, ,
\end{gather*}
on $\Gamma^-$ as a consequence of the Frenet-Serret equations. To obtain an estimate for $u_{\nu\nu}$, it is sufficient to bound $u_{\xi\xi}$ from below by some positive $\tau$. The previously constructed family of solutions \eqref{u(k)(x)} shows that this is not always possible, since \[
\frac{d^2}{ds^2} u^{(k)}(\gamma(s)) = 0 \quad \forall s
\] by the radial symmetry, and $u^{(k)}_\nu \longrightarrow 0$ as $k \longrightarrow \infty$. Therefore $u^{(k)}_{\xi\xi} \longrightarrow 0$ at all points of $\Gamma^-$. Example \eqref{v(k)(x)} shows that this behaviour can also occur in more general settings. \\

What has gone wrong here? A closer examination of the $u^{(k)}$ and of the limiting solution $u^{(\infty)}$ offers us a clue. For $k < \infty$, there is a neighbourhood of $\Gamma^-$ outside of $\Omega$ such that each function can be extended past $\Gamma^-$ into this neighbourhood as a strictly convex $C^{2,\alpha}$ function. In fact, for solutions $u^{(k)}$ with $u^{(k)}_\nu \geq \psi R_-$, the neighbourhood may be taken to be the entire open set enclosed by $\Gamma^-$ containing the origin. In this sense, the quadratic solution, for which the normal derivative equals $\psi R_-$, separates those solutions of \eqref{TheEquation} which also solve the Dirichlet problem on the strictly convex domain enclosed by $\Gamma^+$ from those which do not.

However, if we attempt to extend $u^{(\infty)}$ in a similar way into a neighbourhood outside of $\Gamma^-$, it is not possible. The best we can do is to extend it as a constant function, with the result being merely $C^{1,\alpha}$ and no longer strictly convex. This demonstrates that even in this simple setting, there are choices of $\phi$, namely those less than $\phi^\psi_\infty$ given by \eqref{phiklowerbound}, for which the ellipticity of the equation breaks down. In this sense, $\phi^\psi_\infty$ in \eqref{phiklowerbound} is the ``critical $\phi$" for this choice of $\Omega$.

This reasoning suggests that we should look for the existence of a strictly convex subsolution $\underline{u}$ which acts as a lower barrier and prevents this breakdown of ellipticity. Examining the concentric setting for one final time, if $\psi_0 > \psi$, then from \eqref{phiklowerbound}, the ``critical $\phi$" for the problem with $\psi_0$ is greater than that for the problem with $\psi$. Hence, if there exists a subsolution $\underline{u}$ with $\det[D^2 \underline{u}] = \psi_0 > \psi$ and $\underline{u}_\nu = \gamma_0 \underline{u} + \phi$, then we must also have \[
\phi > \phi^{\psi_0}_\infty > \phi^\psi_\infty \, ,
\] and therefore there is a solution of \eqref{TheEquation} which is strictly convex and with bounded double normal derivative on $\Gamma^-$.

\subsubsection{Proof of Proposition \ref{Double Normal Estimate}}

Let $x_0 \in \Gamma^-$ and choose a unit tangent $\xi$ to $\Gamma^-$ at $x_0$. Assume the existence of $\gamma(s)$, $\underline{u}$ and $\tau$ as in Proposition \ref{Double Normal Estimate}, and define the function $U$ on $\gamma(s)$ by \[
U(s) \eqdef u_\nu(\gamma(s)) - \underline{u}_\nu(\gamma(s)) \, .
\] Then along $\gamma$,
\begin{gather*}
	\frac{d^2}{ds^2} U = \gamma_0 \frac{d^2}{ds^2} (u - \underline{u}) = \gamma_0 \left[ u_{\xi\xi} - \underline{u}_{\xi\xi} - \kappa_\xi (u_\nu - \underline{u}_\nu) \right] = \gamma_0 \left[ u_{\xi\xi} - \underline{u}_{\xi\xi} - \kappa_\xi U \right] \\
	\implies u_{\xi\xi} = \frac{1}{\gamma_0} \frac{d^2}{ds^2} U + \kappa_\xi U + \underline{u}_{\xi\xi} \, .
\end{gather*}
Therefore, by the assumption on the right-hand side of the above equation, $u_{\xi\xi} \geq \tau$ for any tangent vector $\xi$. \\

Using this, and by expanding $u_{\nu\nu}$ in terms of the tangential and mixed second derivatives on the boundary, the bound on $u_{\nu\nu}$ follows from the structure of the Monge-Amp\`ere equation itself and the boundedness of $\psi$. \qed


\subsection{Existence of Solutions for \eqref{TheEquation}}\label{sec-existence}

As usual, the $C^2$ estimates for solutions of \eqref{TheEquation} on $\Omega$ reduce to boundary estimates. This fact follows from a standard maximum principal argument using \eqref{Uijxx} based on \cite{CNS-1}, which we briefly sketch. Define the linear operator $L = F^{ij}D_{ij}$. Then
\begin{gather*}
	Lu_{\xi\xi} = F^{ij}u_{ij\xi\xi} \stackrel{\eqref{Uijxx}}{\geq} \ln(\psi^n)_{\xi\xi} \implies L(u_{\xi\xi} + Ku) \geq \ln(\psi^n)_{\xi\xi} + Kn \geq 0 \, ,
\end{gather*}
provided $K$ is chosen large enough. Therefore by the maximum principle, \[
\sup_{\overline{\Omega}} u_{\xi\xi} - KC_0 \leq \sup_{\overline{\Omega}}(u_{\xi\xi} + Ku) \leq \sup_{\Gamma^+ \cup \Gamma^-} (u_{\xi\xi} + Ku) \leq \sup_{\Gamma^+ \cup \Gamma^-} u_{\xi\xi} + KC_0 \, ,
\] and the result follows. This argument also applies to mixed derivatives since the Monge-Amp\`ere equation is invariant under orthogonal change of coordinates. \\

On $\Gamma^+$, the second derivative estimates follow from those for the Dirichlet problem in \cite{CNS-1}. Therefore, in summary, for $u\in C^4(\Omega)\cap C^3(\overline{\Omega})$ Propositions \ref{th-C2} and \ref{Double Normal Estimate} and equation \eqref{uxinu} imply the global estimate \[
\| u \|_{C^2(\overline{\Omega})} \leq C_5 \, ,
\] where $C_5$ depends on $n,\Omega,\Gamma^\pm,\psi,\gamma_0$ and $\phi$. \\

From the a priori estimate above, we use the method of continuity to show that under the conditions of Theorems \ref{main-th-C01}, \ref{main-th-tan} and \ref{main-th-nor} there exists a strictly convex $u:\Omega \longrightarrow \R$ satisfying \eqref{TheEquation}. First, note that under these conditions, Theorem 3.2 from \cite{LTHJBE} or Section 6 from \cite{HolderEstimates} imply that for any $\alpha\in(0,1)$, \[
\|u\|_{C^{2,\alpha}} \leq K \, ,
\] where $K$ is a constant depending on $n,\Omega,\psi, C_5$ and $\alpha$.

We can now directly apply Theorem 17.28 from \cite{GilbargTrudinger} to conclude that there exists a strictly convex solution $u \in C^{2,\alpha}(\Omega)$ of \eqref{TheEquation} under the conditions of Theorems \ref{main-th-tan} and \ref{main-th-nor}, thereby proving Corollary \ref{Cor-Existence}.


\section{Estimates and Existence for $\psi=\psi(x,u,Du)$}\label{sec-generalRHS}

We now consider the more general Monge-Amp\`ere type equation
\begin{equation*}
\begin{cases}
\det[D^2 u] = \psi^n (x,u,Du)  & \text{ in } \Omega \, , \\
u = 0  & \text{ on } \Gamma^+ \, , \\
u_\nu = \gamma_0 u + \phi(x) & \text{ on } \Gamma^- \, ,
\end{cases}
\end{equation*}
where $\Omega, \Gamma^+, \Gamma^-, \nu, \gamma_0$ and $\phi$ are as before, but now the function $\psi$ potentially depends additionally on $u$ and $Du$. Assume that $\psi_z(x,z,p) \geq 0$. In this section we re-derive a priori $C^0$, $C^1$ and $C^2$ estimates for solutions under extra conditions.


\subsection{The $C^0$ Estimate}\label{sec-C0Du}

\begin{proposition}\label{C0Du}
	Suppose $u\in C^2(\Omega)\cap C^1(\overline{\Omega})$ is a strictly convex solution of \eqref{TheEquationDu}, and that \eqref{Structure} and \eqref{StructureIneq} hold. Then $u$ is bounded, with \[
	0 \leq -u \leq C_0' \eqdef \frac{R_0 + \max_{\Gamma^-}|\phi|}{\gamma_0} + R_0 \diam (\Omega) \, .
	\]
\end{proposition}

\begin{proof}
	We follow (and simplify) the proof of Theorem 2.1 in \cite{LTU}. Let $R_0$ be as in \eqref{R0}. Since $\det[D^2 u] = \psi^n$ is the Jacobian of the gradient map $Du : \Omega \longrightarrow \R^n$, we have for any $R > R_0$, \[
	\int_{Du(\Omega)} h \stackrel{\eqref{Structure}}{\leq} \int_{\Omega} g \stackrel{\eqref{R0}}{<} \int_{|p|<R} h \, .
	\] This implies that there exists $p \in B_R(0) \setminus Du(\Omega)$. We can now define an affine function $\pi$ such that $D\pi = p$, $\pi(x_0) = u(x_0)$ for some $x_0 \in \overline{\Omega}$, and $\pi \leq u$ in $\Omega$. We must have that $x_0 \in \d\Omega$, or else $p$ would be tangent to $u$ at $x_0$, contradicting the definition of $p$.
	
	If $x_0 \in \Gamma^+$ then $\pi(x_0) = u(x_0) = 0$ and since $\pi(x) = \pi(x_0) + D\pi(x_0) \cdot (x - x_0)$, \[
	0 \leq -u(x) \leq \sup_{\overline{\Omega}} |\pi| \leq R\diam(\Omega) \quad \forall x\in \Omega \, .
	\] On the other hand, if $x_0 \in \Gamma^-$ then $-R \leq -|p| \leq \pi_\nu (x_0) \leq u_\nu (x_0)$. This implies \[
	\pi(x_0) = u(x_0) \geq \frac{u_\nu(x_0) - \phi(x_0)}{\gamma_0} \geq \frac{-R - \max_{\Gamma^-}|\phi|}{\gamma_0} \, .
	\] Thus \[
	0 \leq -u(x) \leq -\pi(x) \leq \frac{R + \max_{\Gamma^-}|\phi|}{\gamma_0} + R\diam(\Omega) \, .
	\] The result follows by combining the two cases and letting $R \longrightarrow R_0$.
\end{proof}


\subsection{Another Counterexample}\label{sec-CounterexampleDu}

Under the structure conditions \eqref{Structure} and \eqref{StructureIneq}, it is possible to extend the proof of Proposition \ref{Gradient} to include the case where $\psi = \psi(x,u)$, noting that $\|\psi\|_\infty$ in \eqref{GradientBound} will also depend on $C_0'$.

However, we now present a counterexample demonstrating that in general, when $\psi$ depends additionally on $Du$, solutions of $\eqref{TheEquationDu}$ may have arbitrarily large gradient on $\Gamma^+$. The example also shows that even the structure conditions \eqref{Structure} and \eqref{StructureIneq} are not sufficient to guarantee a gradient estimate in this setting. \\

Consider once more the two dimensional radial problem \eqref{RadialMongeAmpere} with general right-hand side $\psi(x,u,Du)$ satisfying no constraints. Let $R_+ - R_- < 1$. We will show that it is possible to construct a solution $u$ of this equation satisfying the boundary conditions for which $|Du|$ blows up on $\Gamma^+$.

Set \[
u(r) = (R_+ - r)[\ln(R_+ - r) - 1] \, .
\] Differentiating this, we find
\begin{gather*}
u_r(r) = -\ln(R_+ - r) \, , \quad u_{rr}(r) = \frac{1}{R_+ - r} > 0 \, .
\end{gather*}
Therefore, $u$ satisfies the Monge-Amp\`ere type equation
\begin{equation}\label{TheEquationDuk}
\begin{cases}
\det[D^2 u] = \frac{u_{rr}u_r}{r} = \frac{u_r e^{u_r}}{r} = \frac{x\cdot Du}{|x|^2} \exp\left( \frac{x\cdot Du}{|x|} \right)  & \text{ in } \Omega \, , \\
u = 0  & \text{ on } \Gamma^+ \, , \\
u_\nu = \gamma_0 u + \phi & \text{ on } \Gamma^- \, ,
\end{cases}
\end{equation}
with \[
\phi = \gamma_0(R_+ - R_-)[1 - \ln(R_+ - R_-)] -\ln(R_+ - R_-) \, .
\] However, $u_r(r) \longrightarrow \infty$ as $r \longrightarrow R_+$. It is straightforward to construct a sequence of $C^{2,\alpha}$ solutions of \eqref{TheEquationDuk} which converge to $u$ by taking
\begin{gather*}
\phi_k = d_k + \gamma_0 \Bigl\{ e^{-d_k}(d_k + 1) + (e^{-d_k} + R_- - R_+)[\ln(e^{-d_k} + R_- - R_+) - 1] \Bigr\} \, , \\
u^{(k)}(r) = (e^{-d_k} + R_- - r)[\ln(e^{-d_k} + R_- - r) - 1] - (e^{-d_k} + R_- - R_+)[\ln(e^{-d_k} + R_- - R_+) - 1] \, .
\end{gather*}
where $d_k$ is a sequence of positive numbers converging from below to $-\ln(R_+ - R_-)$. \\

However, this problem satisfies the structure conditions \eqref{Structure} and \eqref{StructureIneq}. By defining \[
g(x) = \frac{1}{|x|} \, , \quad h(p) = \frac{e^{-|p|}}{|p|} \, ,
\] we see that \[
\psi^2(x,z,p) \leq \frac{|p|}{|x|}e^{|p|} = \frac{g(x)}{h(p)} \, .
\] Moreover, if $\Omega$ has the property that $R_+ - R_- < 1$, then \[
\int_{\Omega} g(x) \mathrm{d}x = 2\pi (R_+ - R_-) < 2\pi = \int_{\R^2} h(p) \mathrm{d}p \, .
\] Therefore the structure conditions are not sufficient to guarantee a gradient estimate.


\subsection{The $C^1$ Estimate}\label{sec-C1Du}

In this section we prove local and global gradient estimates for $u\in C^2(\Omega)\cap  C^1(\overline{\Omega})$. With Condition \ref{StructureConditions} alone it is possible only to prove a local gradient estimate on a subset of $\Omega$. In particular, there exists a neighbourhood of $\Gamma^-$ in which the gradient is bounded. To prove the global gradient estimate we must additionally impose Structure Condition \ref{StructureConditionGradient}.

\begin{proposition}\label{GradientDu}
	Suppose $u \in C^2(\Omega)\cap C^1(\overline{\Omega})$ is a strictly convex solution of \eqref{TheEquationDu}. Suppose Condition \ref{StructureConditions} holds. For $0 > -\lambda \geq \sup_{\Gamma^-} u$, define the set \[
	\Omega_\lambda \eqdef \{ x\in\Omega : u(x) \leq -\lambda \} \subset \Omega \, .
	\] Then
	\begin{equation*}
	\sup_{\overline{\Omega_\lambda}} |Du| \leq C_{1,\rm{loc}}' \eqdef \frac{C_0'}{\dist(\Omega_\lambda,\Gamma^+)} \, .
	\end{equation*}
\end{proposition}

\begin{proof}
	The set $\d\Omega_\lambda$ comprises $\Gamma^-$ as an inner boundary, and a convex outer boundary, \[
	\Gamma^+_\lambda \eqdef \{ x\in\Omega : u(x) = - \lambda \} \, .
	\] By the convexity of $u$, it is enough to prove that $\sup_{\Gamma^+_\lambda}|Du| \leq C_0'/\dist(\Omega_\lambda,\Gamma^+)$. We first show that $|Du|$ is bounded by a constant on $\Gamma^+_\lambda$. Define $\normal_\lambda$ to be the unit normal field to $\Gamma^+_\lambda$ in the direction pointing outwards from $\Omega_\lambda$. Suppose for contradiction that there exists $x_0 \in \Gamma^+_\lambda$ such that $|u_{\normal_\lambda}(x_0)| = \infty$. Let $\ell_\lambda$ be the line segment starting at $x_0$ in the direction $\normal_\lambda(x_0)$, and let $y_0$ be the unique point where $\ell_\lambda$ intersects with $\Gamma^+$. Then, by convexity, \[
	0 = u(y_0) \geq u(x_0) + Du \cdot (y_0 - x_0) = -\lambda + |u_{\normal_\lambda}(x_0)||y_0 - x_0| \, ,
	\] which is impossible unless $|y_0 - x_0| = 0$. However, this contradicts the definitions of $\Omega_\lambda$ and $x_0$, and the continuity of $u$. Therefore there exists a constant $\beta$ such that $|u_{\normal_\lambda}(x_0)| \leq \beta$.
	
	Finally, again by convexity of $u$, we know that for each $x_0 \in \Gamma^+_\lambda$, \[
	|Du(x_0)| \leq \frac{-u(x_0)}{\dist(x_0,\Gamma^+)} \leq \frac{C_0'}{\dist(\Omega_\lambda,\Gamma^+)} \, ,
	\] which implies the result by taking the supremum over $x_0 \in \Gamma^+_\lambda$.
\end{proof}

Under further structure conditions however, we can improve upon Proposition \ref{GradientDu} with a global gradient estimate.

\begin{proposition}\label{C1Du}
	Suppose $u \in C^2(\Omega)\cap C^1(\overline{\Omega})$ is a strictly convex solution of \eqref{TheEquationDu}. Suppose Conditions \ref{StructureConditions} and \ref{StructureConditionGradient} hold. Then there exists a constant $C_1' = C_1'(n,\Omega,\Gamma^+,\beta,Z)$ such that \[
	\sup_{\overline{\Omega}} |Du| \leq C_1' \, .
	\]
\end{proposition}

\begin{proof}
	Since $u$ is convex, we need only prove a gradient estimate on $\Gamma^+$. This follows directly from the proof of Theorem 17.21 in Chapter 17 of \cite{GilbargTrudinger}.
\end{proof}


\subsection{The $C^2$ Estimate}\label{sec-C2Du}

We now show that it is possible to obtain a $C^2$ estimate for $u$. As before, this reduces to proving second order tangential and normal derivative estimates on $\Gamma^-$. The proof of the tangential derivative estimate is essentially the same as Section \ref{sec-tan}, but with a more general barrier. The proof of the normal derivative estimate is exactly the same as Section \ref{sec-normal}, since this proof does not depend on $\psi$.

\begin{proposition}\label{th-C2Du}
	Suppose $u \in C^4(\Omega)\cap C^3(\overline{\Omega})$ is a strictly convex solution of \eqref{TheEquationDu}. Then there exists a constant $C_2' = C_2'(n,\Omega,\Gamma^\pm,\psi,\phi)$ such that \[
	\sup_{\overline{\Omega}}|D^2 u| \leq C_2' \, ,
	\] provided that there exists a subsolution satisfying Condition \ref{SubsolutionCondition}, and \[
	2\max_{\substack{i = 1,...,n-1 \\ \bar{x} \in\Gamma^-}} \kappa_i(\bar{x}) + \tilde{C} < \gamma_0 + \max \left\{ 0 , \min_{\Gamma^-}\frac{\gamma_0 u + \phi}{\tilde{M} + (1-N^4)u} \right\} \, ,
	\] where $\kappa_i(\bar{x})$ is the $i^{\text{th}}$ principal curvature of $\Gamma^-$ at $\bar{x}$, $N = N(n,\Omega,\psi,\phi)$, $\tilde{C} = \tilde{C}(n,\Omega,\Gamma^\pm,\psi,\phi,N)$ and $\tilde{M} = \tilde{M}(n,\Omega,\Gamma^\pm,\psi,\phi)$ are constants.
\end{proposition}

In order to obtain the double normal estimate on $\Gamma^-$ we assume the existence of a subsolution satisfying Condition \ref{SubsolutionCondition}. This subsolution also forces the solution to be bounded and its gradient to be bounded on $\Gamma^+$, and thus globally bounded on $\overline{\Omega}$ by the maximum principle. Therefore, \[
\sup_{\overline{\Omega}} |Du| \leq C_1' \, .
\]

We also note that the estimate for $|u_{\xi\nu}|$ on $\Gamma^-$, where $\xi$ is tangential, of \eqref{uxinu} is still valid for $\psi = \psi(x,u,Du)$. Moreover, the proof of Proposition \ref{Double Normal Estimate} is also valid. Thus under the assumptions of \ref{th-C2Du},
\begin{equation*}
	\sup_{\Gamma^-} |D_{\nu\nu}| \leq C_4' \, .
\end{equation*}
Therefore it remains only to bound $|u_{\xi\xi}|$ above, where $\xi$ is a unit tangent vector to $\Gamma^-$.

As in Section \ref{sec-tan}, it is convenient to denote the extension of the Neumann condition into $\Omega$ by $\chi(x, u)=\gamma_0 u + \phi(x)$, where by an abuse of notation $\phi$ is a smooth extension of its counterpart on the boundary. Let us consider a modification of the auxiliary function used in Section \ref{sec-tan}, \[
\tilde{w} = e^{N|Du|^2} w = e^{N|Du|^2} [g(u)u_{\xi\xi} + a_k u_k + b + M|x|^2] \, ,
\] where $w$ is the function from Section \ref{sec-tan}, $\xi$ is an arbitrary direction, $g(u)$ is a nonnegative function to be chosen, $M, N$ are constants also to be chosen, and $a_k$ and $b$ are as in \cite{LTU}:
\begin{gather*}
a_k = 2(\xi\cdot \nu)(\chi_u\xi_k'-\xi_k'D_i\nu_k) \, , \\
b = 2(\xi\cdot \nu)\xi_k'\chi_{x_k} \, ,
\end{gather*}
$\xi'=\xi-(\xi\cdot\nu)\nu$ with $\nu$ a $C^{2,1}(\Omega)$ extension of the inner unit normal field on $\Gamma^-$. \\

Suppose for contradiction that $w$ has a local maximum at $x_0 \in \Omega$. Then at $x_0$ we have
\begin{gather}
\frac{\tilde{w}_i}{\tilde{w}} = 2N u_k u_{ki} + \frac{w_i}{w} = 2N u_k u_{ki} + \frac{1}{w} (g' u_i u_{\xi\xi} + g u_{i\xi\xi} + a_{k,i}u_k + a_ku_{ki} + b_i + 2Mx_i) = 0 \, , \label{C2 wiDu} \\
\frac{\tilde{w}_{ij}}{\tilde{w}} - \frac{\tilde{w}_i\tilde{w}_j}{\tilde{w}^2} = 2N u_{ki} u_{kj} + 2N u_k u_{kij} + \frac{w_{ij}}{w} - \frac{w_i w_j}{w^2} \leq 0 \, . \label{C2 wijDu}
\end{gather}
For clarity, we use a comma in the subscript to denote the derivatives of the components $a_k$. As in Section \ref{sec-tan}, set $F^{ij} = u^{ij}$. \\

We want to multiply equation \eqref{C2 wijDu} by $F^{ij}$ and eliminate the highest order $\xi$-derivative terms using \eqref{C2 wiDu}. If $F^{ij}$ and $u_{ij}$ are as above, then for any direction $\xi$ we have
\begin{gather*}
F^{ij}u_{ij\xi} = \ln(\psi^n)_\xi \, , \\
F^{ij}u_{ij\xi\xi} = \ln(\psi^n)_{\xi\xi} + F^{ij}u_{jk\xi}F^{kl}u_{li\xi} \, .
\end{gather*}
Also required are the expressions for $\ln(\psi^n)_k$ and $\ln(\psi^n)_{\xi\xi}$,
\begin{gather*}
\ln(\psi^n)_k = \frac{n}{\psi} (\psi_{x_k} + \psi_z u_k + \psi_{p_m}u_{km}) \, , \\
\ln(\psi^n)_{\xi\xi} = \frac{n}{\psi} (\psi_{\xi\xi} + 2\psi_{\xi z} u_\xi + 2\psi_{\xi p_m} u_{m\xi} + \psi_z u_{\xi\xi} + \psi_{zz}u_\xi^2 + 2\psi_{zp_m}u_\xi u_{m\xi} + \psi_{p_m}u_{m\xi\xi} + \psi_{p_m p_q} u_{m\xi} u_{q\xi}) \\
- \frac{n}{\psi^2} (\psi_{\xi} + \psi_z u_\xi + \psi_{p_m}u_{m\xi})(\psi_{\xi} + \psi_z u_\xi + \psi_{p_q}u_{q\xi}) \, .
\end{gather*}

\begin{lemma}\label{FijwijDu}
	From $b$, define $\overline{b_k} = 2(\xi\cdot \nu)\xi_k'$, and $B_i = a_{k,i}u_k + b_i + 2Mx_i$. Set $\mathcal{T} = \sum F^{ii}$. With this notation, we have
	\begin{subequations}
		\begin{gather}
		0 \geq \frac{F^{ij}\tilde{w}_{ij}}{e^{N|Du|^2}} = \nonumber \\
		2Nw\left[ \Delta u + \frac{n}{\psi} u_k (\psi_{x_k} + \psi_z u_k) \right] + \left[ ng' + \left( g'' - \frac{2(g')^2}{g} \right)F^{ij}u_i u_j \right] u_{\xi\xi} + g F^{ij}u_{jk\xi}F^{kl}u_{li\xi} \label{Good} \\
		- \frac{F^{ij}}{w}\left( g^2u_{i\xi\xi}u_{j\xi\xi} + 2gg'u_{\xi\xi} u_{i\xi\xi}u_j \right) - \frac{2g}{w}a_k u_{k\xi\xi} - \frac{2g}{w} F^{ij} u_{i\xi\xi}B_{j} \label{3rd} \\
		- \frac{(g')^2}{w}u_{\xi\xi}^2 F^{ij}u_iu_j + \frac{ng}{\psi} \left[ \psi_{p_m p_q} u_{m\xi}u_{q\xi} - \frac{(\psi_{p_m}u_{m\xi})^2}{\psi} \right] \label{2-2nd} \\
		- \frac{2g'}{w} u_{\xi\xi} \left( a_k u_k + F^{ij} B_i u_j \right) - \frac{n}{\psi}\psi_{p_m} \left[ g' u_m u_{\xi\xi} + 2\frac{g}{\psi}(\psi_\xi + \psi_z u_\xi) u_{m\xi} - \gamma_0\overline{b_k} u_{km} \right] \label{1-2nd-a} \\
		+ 2(\psi_{\xi p_m} + \psi_{zp_m}u_\xi) u_{m\xi} - \frac{1}{w}a_k a_q u_{kq} - \frac{4Nwg'}{g}|Du|^2 \label{1-2nd-b} \\
		+ \frac{ng}{\psi}\left[ \psi_{\xi\xi} + 2\psi_{\xi z}u_\xi + \psi_{zz}u_\xi^2 - \frac{(\psi_\xi + \psi_z u_\xi)^2}{\psi} \right] + \frac{n}{\psi}(\psi_{x_k} + \psi_z u_k)(a_k + \gamma_0 \overline{b_k}) - \frac{n}{\psi}\psi_{p_m} B_m \label{1st-a} \\
		- \frac{F^{ij}}{w} B_i B_j - \frac{2g'}{g} \left( F^{ij}u_i B_j + a_k u_k \right) + F^{ij}\Bigl[ a_{k,ij}u_k + \overline{b_k}_{,ij}(\gamma_0 u_k + \phi_k) + 2\overline{b_k}_{,i} \phi_{kj} + \overline{b_k}\phi_{ijk} \Bigr] \label{1st-b} \\
		- \frac{2}{w} a_k B_k + 2\sum (a_{k,k} + \gamma_0 \overline{b_k}_{,k}) + 2M\mathcal{T} \, . \label{1st-c}
		\end{gather}
	\end{subequations}
\end{lemma}

\begin{proof}
	The proof is simple but lengthy, making use of the same methods as Lemma \ref{Fijwij}.
\end{proof}

We will now consider each of the lines \eqref{Good} to \eqref{1st-b} in Lemma \ref{FijwijDu} individually, beginning with \eqref{Good}. Let $\varepsilon\in(0,1)$. It may be assumed, as in \cite{LTU} that \[
\sup_{\Omega\times\mathbb{S}^{n-1}} (a_k u_k + b + M|x|^2) \leq \frac{\varepsilon}{2}gu_{\xi\xi} \, ,
\] otherwise we would have successfully found an upper bound for $u_{\xi\xi}$. This in turn implies
\begin{equation}\label{wBound}
\frac{gu_{\xi\xi}}{1 + \varepsilon} < \left( 1 - \frac{\varepsilon}{2} \right) gu_{\xi\xi} \leq w < (1 + \varepsilon)gu_{\xi\xi} \leq (1 + \varepsilon)gu_{\eta\eta}  \, ,
\end{equation}
and since $w(x_0,\xi)$ is a maximum point,
\begin{equation}\label{uetaetaBound}
w(x_0,\xi) \geq w(x_0,\eta) \implies u_{\eta\eta} \leq (1 + \varepsilon)u_{\xi\xi} \, ,
\end{equation}
where $\eta$ is the unit vector such that $u_{\eta\eta}$ is the maximum eigenvalue of $D^2u(x_0)$. The upshot of \eqref{wBound} and \eqref{uetaetaBound} is that
\begin{equation}\label{GoodImproved}
2N w \Delta u \geq \frac{2Ngu_{\eta\eta}^2}{(1 + \varepsilon)^2} \geq \frac{Ngu_{\eta\eta}^2}{2} \, .
\end{equation}

Next, to deal with \eqref{3rd}, we note that \eqref{wBound} and \eqref{uetaetaBound} also imply
\begin{equation}\label{wchain}
\frac{g}{w} < \frac{1 + \varepsilon}{u_{\xi\xi}} < \frac{(1 + \varepsilon)^2}{u_{\eta\eta}} < \frac{1 + 3\varepsilon}{u_{\eta\eta}} \, .
\end{equation}
Applying \eqref{wchain} and the Cauchy-Schwarz and Young inequalities to \eqref{3rd}, we see that
\begin{gather}
\left( F^{ij}\frac{w_iw_j}{w} = \right) \ \frac{F^{ij}}{w}\left( g^2u_{i\xi\xi}u_{j\xi\xi} + 2gg'u_{\xi\xi} u_{i\xi\xi}u_j \right) + \frac{2g}{w}a_k u_{k\xi\xi} + \frac{2g}{w} F^{ij} u_{i\xi\xi}B_{j} \nonumber \\
\leq \frac{g}{u_{\eta\eta}}(1 + 3\varepsilon)^2 F^{ij}u_{i\xi\xi}u_{j\xi\xi} + \frac{a_ka_qu_{kq}}{\varepsilon w} + \frac{F^{ij}B_iB_j}{\varepsilon w} + \frac{1+\varepsilon}{\varepsilon} \frac{(g')^2}{g} u_{\xi\xi} F^{ij}u_iu_j \label{Ineq1} \, .
\end{gather}
Moreover, we know from \eqref{C2 wiDu} that
\begin{equation}\label{Ineq2}
F^{ij}\frac{w_iw_j}{w} = 4N^2 w u_ku_qu_{kq} \, .
\end{equation}
Taking $\varepsilon(6+9\varepsilon)$ times \eqref{Ineq2} and combining with \eqref{Ineq1} gives us
\begin{gather*}
(1 + 3\varepsilon)^2 F^{ij} \frac{w_i w_j}{w} \leq \frac{g}{u_{\eta\eta}}(1 + 3\varepsilon)^2 F^{ij}u_{i\xi\xi}u_{j\xi\xi} + \frac{a_ka_qu_{kq}}{\varepsilon w} + \frac{F^{ij}B_iB_j}{\varepsilon w} + \frac{1+\varepsilon}{\varepsilon} \frac{(g')^2}{g} u_{\xi\xi} F^{ij}u_iu_j \\
+ 4N^2 \varepsilon(6 + 9\varepsilon) w u_ku_qu_{kq} \, ,
\end{gather*}
and putting everything together,
\begin{gather}\label{3rdImproved}
\frac{F^{ij}}{w}\left( g^2u_{i\xi\xi}u_{j\xi\xi} + 2gg'u_{\xi\xi} u_{i\xi\xi}u_j \right) + \frac{2g}{w}a_k u_{k\xi\xi} + \frac{2g}{w} F^{ij} u_{i\xi\xi}B_{j} \\
\leq \frac{g}{u_{\eta\eta}} F^{ij}u_{i\xi\xi}u_{j\xi\xi} + 4N^2\frac{\varepsilon(6 + 9\varepsilon)}{(1 + 3\varepsilon)^2} w u_ku_qu_{kq} + \frac{a_ka_qu_{kq}}{\varepsilon^2 (6 + 9\varepsilon) w} + \frac{F^{ij}B_iB_j}{\varepsilon^2 (6 + 9\varepsilon) w} + \frac{1+\varepsilon}{\varepsilon^2 (6 + 9\varepsilon)} \frac{(g')^2}{g} u_{\xi\xi} F^{ij}u_iu_j \nonumber \, ,
\end{gather}
as in \cite{LTU}. \\

Next we deal with \eqref{2-2nd}. Applying $w \geq gu_{\xi\xi}$ (provided $M$ is large enough),
\begin{equation}\label{2-2ndImproved1}
\frac{(g')^2}{w}u_{\xi\xi}^2 F^{ij}u_iu_j \leq \frac{(g')^2}{g}u_{\xi\xi} F^{ij}u_iu_j \, ,
\end{equation}
which will be absorbed into a term in \eqref{Good}. Further,
\begin{equation}\label{2-2ndImproved2}
\frac{ng}{\psi} \left[ \psi_{p_m p_q} u_{m\xi}u_{q\xi} - \frac{(\psi_{p_m}u_{m\xi})^2}{\psi} \right] \leq \frac{n^3g}{\psi} \left( \|D^2_p\psi\|_\infty + \frac{\|D_p\psi\|_\infty^2}{\psi} \right) u_{\eta\eta}^2 \leq C(n,\psi) gu_{\eta\eta}^2 \, .
\end{equation}

We consider \eqref{1-2nd-a} and \eqref{1-2nd-b} together. Using $w \geq gu_{\xi\xi}$ and \eqref{wchain} again, we have
\begin{gather}
\frac{2g'}{w} u_{\xi\xi} \left( a_k u_k + F^{ij} B_i u_j \right) + \frac{n}{\psi}\psi_{p_m} \left[ g' u_m u_{\xi\xi} - 2\frac{g}{\psi}(\psi_\xi + \psi_z u_\xi) u_{m\xi} - \gamma_0\overline{b_k} u_{km} \right]  \nonumber \\
- 2(\psi_{\xi p_m} + \psi_{zp_m}u_\xi) u_{m\xi} + \frac{1}{w}a_k a_q u_{kq} - \frac{4Nwg'}{g}|Du|^2 \nonumber \\
\leq \frac{2g'}{g} \left( C_1'\|a\|_\infty + C_1' \|B\|_\infty \mathcal{T} \right) + C(n,\psi,\gamma_0,C_1')(Ng'+g + 1)u_{\eta\eta} + \frac{1 + 3\varepsilon}{g}C(n,a) \, . \label{1-2ndImproved}
\end{gather}

Finally, the terms of \eqref{1st-a}, \eqref{1st-b} and \eqref{1st-c} are bounded thus:
\begin{gather}
\eqref{1st-a} + \eqref{1st-b} + \eqref{1st-c} \nonumber \\
\geq - C(n,\psi,C_1')g - C(n,\psi,\gamma_0,\phi,C_1',a,Da,D\overline{b_k}) - \frac{1}{w}(2a_kB_k + F^{ij}B_iB_j) \nonumber \\
- \frac{2g'}{g} \left[ C_1'\|a\|_\infty + C_1' \|B\|_\infty \mathcal{T} \right] + \Bigl[ 2M - C(\gamma_0,\phi,C_1',D^2a,\overline{b_k},D\overline{b_k},D^2\overline{b_k}) \Bigr]\mathcal{T} \, . \label{1stImproved}
\end{gather}

To deal with the 3rd derivative terms, we reason as in \cite{LTU} that \[
F^{ij}u_{i\xi\xi}u_{j\xi\xi} = F^{ij}u_{ik\xi}u_{jl\xi}\xi_k\xi_l \leq \Lambda \, ,
\] where $\Lambda$ is the maximum eigenvalue of the matrix $F^{ij}u_{ik\xi}u_{jl\xi}$. Writing $F^{\eta\eta} = 1/u_{\eta\eta}$ and noting that the trace of a positive definite matrix dominates its individual eigenvalues, we have
\begin{equation}\label{3rdImproved!}
g F^{ij}u_{jk\xi}F^{kl}u_{li\xi} \geq gF^{\eta\eta}\Lambda \geq \frac{g}{u_{\eta\eta}} F^{ij}u_{i\xi\xi}u_{j\xi\xi} \, .
\end{equation}

Putting \eqref{GoodImproved}, \eqref{3rdImproved}, \eqref{2-2ndImproved1}, \eqref{2-2ndImproved2}, \eqref{1-2ndImproved}, \eqref{1stImproved} and \eqref{3rdImproved!} together and choosing $\varepsilon = 1/N^2$, $N>1$, we find
\begin{gather*}
0 \geq \frac{F^{ij}\tilde{w}_{ij}}{e^{N|Du|^2}} \geq \frac{Ngu_{\eta\eta}^2}{2} - 4Ngu_{\eta\eta} \frac{n}{\psi} u_k (\psi_{x_k} + \psi_z u_k) + \left[ ng' + \left( g'' - N^4 \frac{(g')^2}{g} \right)F^{ij}u_i u_j \right] u_{\xi\xi} \\
- 4N^2\frac{\varepsilon(6 + 9\varepsilon)}{(1 + 3\varepsilon)^2} w u_ku_qu_{kq} - \frac{a_ka_qu_{kq}}{\varepsilon^2(6 + 9\varepsilon) w} - \frac{F^{ij}B_iB_j}{\varepsilon^2(6 + 9\varepsilon) w} - C(n,\psi) gu_{\eta\eta}^2 \\
- C(n,\psi,\gamma_0,C_1')(g'+g + 1)u_{\eta\eta} - \frac{4g'}{g} \left( C_1'\|a\|_\infty + C_1' \|B\|_\infty \mathcal{T} \right) - \frac{1 + 3\varepsilon}{g}C(n,a) - C(n,\psi,C_1')g \\
- C(n,\psi,\gamma_0,\phi,C_1',a,Da) - \frac{1}{w}(2a_kB_k + F^{ij}B_iB_j) + \Bigl[ 2M - C(\gamma_0,\phi,C_1',D^2a,\overline{b_k},D\overline{b_k},D^2\overline{b_k}) \Bigr]\mathcal{T} \\
\geq \frac{1}{2}Ngu_{\eta\eta}\left( u_{\eta\eta} - C(n,\psi,C_1') \right) + \left[ ng' + \left( g'' - \frac{1 + 4\varepsilon}{\varepsilon} \frac{(g')^2}{g} \right)F^{ij}u_i u_j \right] u_{\xi\xi} - C(n,\psi,C_1')g u_{\eta\eta}^2 \\
- C(n,\psi,\gamma_0,C_1')(Ng'+g + 1)u_{\eta\eta} - \frac{N^2C(n,a)}{g} - \frac{4g'}{g} C_1'\|a\|_\infty - C(n,\psi,C_1')g - C(n,\psi,\gamma_0,\phi,C_1',a,Da) \\
- \frac{C(n,a) \|B\|_{\infty}}{gu_{\eta\eta}} + \left( 2M - C(\gamma_0,\phi,C_1',D^2a,\overline{b_k},D\overline{b_k},D^2\overline{b_k}) - \frac{4g'}{g}C_1'\|B\|_{\infty} - \frac{C(n) N^2 \|B\|_{\infty}^2}{gu_{\eta\eta}} \right)\mathcal{T} > 0 \, ,
\end{gather*}
provided $u_{\eta\eta} > C(n,\psi,C_1')$ and that we choose $M$ and $N$ large enough. Notice that we don't need to make an explicit choice for $g$ to obtain this estimate, although it is convenient to choose
\begin{equation}\label{gChoice}
g(u) = \sqrt[1-N^4]{\tilde{M} + (1 - N^4)u} \, ,
\end{equation}
with $\tilde{M}$ possibly different from $M$. \\

The preceding argument demonstrates that the maximum of $\tilde{w}$ occurs on the boundary $\d\Omega$. Assuming that the maximum occurs at $x_0\in\Gamma^-$, we have at $x_0$,
\begin{gather*}
0 \geq \frac{\tilde{w}_\nu}{e^{N|Du|^2}} = 2Nwu_ku_{k\nu} + gu_{\nu\xi\xi} + g'u_\nu u_{\xi\xi} + a_{k,\nu}u_k + b_\nu + 2M\inprod{x}{\nu} \\
\implies u_{\nu\xi\xi} \leq \left( 4NC(n,C_1',C_3,C_4') - \frac{g'}{g}u_{\nu} \right) u_{\xi\xi} + O(1) \, ,
\end{gather*} since $a_k = 0$ when $\xi\perp\nu$. Now, \eqref{unuxixi} implies
\begin{gather*}
\kappa_\xi u_{\nu\nu} + (\chi_u - 2\kappa_\xi)u_{\xi\xi} + O(1) = u_{\nu\xi\xi} \leq \left( 4NC(n,C_1',C_3,C_4') - \frac{g'}{g}u_{\nu} \right) u_{\xi\xi} + O(1) \\
\implies \kappa_\xi u_{\nu\nu} + \left( \chi_u - 2\kappa_\xi - 4NC(n,C_1',C_3,C_4') + \frac{g'}{g}\chi \right)u_{\xi\xi} \leq O(1) \, .
\end{gather*}
Therefore, Proposition \ref{th-C2Du} holds provided \[
2\kappa_\xi + 4NC(n,C_1',C_3,C_4') < \chi_u + \frac{g'}{g}\chi \, .
\] In particular, if we choose $g$ according to \eqref{gChoice}, Proposition \ref{th-C2Du} holds provided \[
2\kappa_\xi + 4NC(n,C_1',C_3,C_4') < \gamma_0 + \max\left\{ 0 , \min_{\Gamma^-} \frac{\gamma_0 u + \phi}{\tilde{M} + (1 - N^4)u} \right\} \, ,
\] analogously to Proposition \ref{th-C2}.

\begin{remark}\label{psi(x,u)}
	If $\psi = \psi(x,u)$ only, then we can make the calculation substantially simpler by noticing that the worst 2nd derivative terms in equations \eqref{2-2nd} to \eqref{1st-c} disappear when we set to zero any derivatives of $\psi$ with respect to the $p$ variables. We can therefore use the simpler barrier function $w$ from Proposition \ref{th-C2}. The expression from Lemma \ref{Fijwij} therefore becomes
	\begin{gather*}
	F^{ij}w_{ij} = \left[ ng' + \left( g'' - \frac{2(g')^2}{g} \right)F^{ij}u_i u_j \right]u_{\xi\xi} - \frac{2g'}{g}F^{ij}u_i \left( u_k a_{k,j} + b_j + 2Mx_j \right) - \frac{2g'}{g} a_k u_k \\
	+ g\left( \frac{n\psi_z}{\psi}u_{\xi\xi} + O(1) + F^{ij}u_{jk\xi}F^{kl}u_{li\xi} \right) + F^{ij}\Bigl[ a_{k,ij}u_k + \overline{b_k}_{,ij}(\gamma_0 u_k + \phi_k) + 2\overline{b_k}_{,i} \phi_{kj} + \overline{b_k}\phi_{ijk} \Bigr] \\
	+ (a_k + \gamma_0 \overline{b_k}) \ln(\psi^n)_k + 2\sum (a_{k,k} + \gamma_0 \overline{b_k}_{,k}) + 2M\sum F^{ii} \, ,
	\end{gather*}
	where the $\frac{n\psi_z}{\psi}u_{\xi\xi} + O(1)$ terms are those arising from $\ln(\psi^n)_{\xi\xi}$. Since $\psi_z \geq 0$ the extra second derivative term is positive and hence can be ignored. The other extra terms are all bounded since they depend only on $n, \psi, u$ and $Du$. Therefore the tangential $C^2$ estimate holds under the weaker curvature condition \ref{CurvatureCondition}, at the expense of $M$ merely being a slightly larger constant.
\end{remark}

\begin{remark}
	As in Section \ref{sec-existence}, if the maximum of $w$ occurs on $\Gamma^+$ then we can estimate the derivatives in the standard way, see \cite{CNS-1}. The subsolution required for the double normal estimate is also sufficient to allow us to apply the method of \cite{CNS-1} to bound $|D^2 u|$ on $\Gamma^+$, which itself requires the existence of a far less complicated subsolution.
\end{remark}

\subsection{Existence of Solutions for \eqref{TheEquationDu}}

The existence of solutions now follows in the same way as in \ref{sec-existence}. The argument of Section 7(d) of \cite{CNS-1} using the auxiliary function defined in equations (7.30) and (7.31) is applicable to our case, showing that the interior $C^2$ estimates reduce to boundary estimates. Therefore, from the a priori estimates above, we can employ the method of continuity to show that under the conditions of Propositions \ref{C0Du}, \ref{C1Du} and \ref{th-C2Du} there exists a strictly convex $u:\Omega \longrightarrow \R$ satisfying \eqref{TheEquationDu}. First, note that under these conditions, Theorem 3.2 from \cite{LTHJBE} or Section 6 from \cite{HolderEstimates} imply that for any $\alpha\in(0,1)$, \[
\|u\|_{C^{2,\alpha}} \leq K \, ,
\] where $K$ is a constant depending on $n,\Omega,\psi$ and $C_2'$.

We can now directly apply Theorem 17.28 from \cite{GilbargTrudinger} to conclude that there exists a strictly convex solution $u \in C^{2,\alpha}(\Omega)$ of \eqref{TheEquationDu}. Thus Corollary \ref{Cor-ExistenceDu} is proved.


\section{The Equation of Prescribed Gauss Curvature}\label{sec-PrescribedGC}

The equation of prescribed Gauss curvature \eqref{PrescribedGauss} falls into the category of equations for which we can prove existence of solutions using the results of the previous section. However, due to the form of the right-hand side, we are able to simplify the proofs substantially. In particular, because the function \[
\psi(x,z,p) = K^{1/n}(x) (1+|p|^2)^{(n+2)/2n}
\] is convex in the $p$ variables, we can eliminate the 2nd-derivative-squared terms from the square bracket in \eqref{2-2nd} using a trick, meaning we can start the process with a better auxiliary function, as mentioned in \cite{LTU} and avoid settling for the nasty bound in \eqref{2-2ndImproved2}. The upshot of this is that the resulting condition on the curvatures of $\Gamma^-$ is much weaker than in the full generality of Proposition \ref{th-C2Du}.

The results of this section apply to more general $\psi(x,u,Du)$ which are convex with respect to $Du$. However, we focus here purely on the equation of prescribed Gauss curvature.

\begin{proposition}
	Suppose $u \in C^4(\Omega)\cap C^3(\overline{\Omega})$ is a strictly convex solution of \eqref{PrescribedGauss}. Suppose that Condition \ref{PGCondition} holds, and that there exists a subsolution satisfying Condition \ref{SubsolutionCondition}. Then there exist constants $C^G_0, C^G_1, C^G_2$ depending on $n,\Omega,\Gamma^\pm,\psi$ and $\phi$ such that \[
	\sup_{\overline{\Omega}}|u| \leq C^G_0 \, , \quad\sup_{\overline{\Omega}}|D u| \leq C^G_1 \, , \quad \sup_{\overline{\Omega}}|D^2 u| \leq C^G_2 \, ,
	\] provided that \[
	2\max_{\substack{i = 1,...,n-1 \\ \bar{x} \in\Gamma^-}} \kappa_i(\bar{x}) < \gamma_0 + \max \left\{ 0 , \min_{\Gamma^-}\frac{\gamma_0 u + \phi}{M-u} \right\} \, ,
	\] where $\kappa_i(\bar{x})$ is the $i^{\text{th}}$ principal curvature of $\Gamma^-$ at $\bar{x}$, and $M = M(n,\Omega,\Gamma^\pm,\psi,\phi)$ is a constant.
\end{proposition}

\begin{proof}
	For the $C^0$ estimate, \eqref{GClessthannball} from Condition \ref{PGCondition} suffices since we may take $g(x) = K(x)$ and $h(p) = (1 + |p|)^{-(n+2)/2}$ in Condition \ref{StructureConditions}.
	
	Similarly, \eqref{K=0} from Condition \ref{PGCondition} suffices to prove the global gradient bound, as in \cite{TUDirichletPGC}. This is because if $K$ is a smooth function with $K = 0$ on $\Gamma^+$, there is a neighbourhood of $\Gamma^+$ such that $\eqref{StructureGradient}$ holds with $\beta = 1$ and $Z$ being a sufficiently large constant function.
	
	The key point in the $C^2$ estimate is as follows. Taking a cue from \cite{TUDirichletPGC}, consider the equivalent form of \eqref{PrescribedGauss}, \[
	F[D^2u] = \det[D^2u]^{1/n} = \psi = K^{1/n} (1 + |Du|^2)^{(n+2)/2n} \, .
	\] Define \[
	F^{ij} = \frac{1}{n} \det[D^2u]^{1/n} u^{ij} = \frac{\psi}{n} u^{ij} \, ,
	\] so that $F^{ij}u_{ij} = \psi$. By Jacobi's formula and the same process as in the derivation of \eqref{Uijxx},
	\begin{gather}
	F^{ij}_\xi = \frac{\psi_\xi}{\psi}F^{ij} - \frac{n}{\psi}F^{ik}F^{jl}u_{kl} \, , \nonumber \\
	F^{ij}u_{ij\xi\xi} = \psi_{\xi\xi} + \frac{1}{\psi} \left[ nF^{ik}F^{jl}u_{ij\xi}u_{kl\xi} - (F^{ij}u_{ij\xi})^2 \right] \, . \label{H2LessnA2}
	\end{gather}
	The last term on the right-hand side of \eqref{H2LessnA2} is nonnegative by the Cauchy-Schwarz inequality. Thus \[
	F^{ij}u_{ij\xi\xi} \geq (1 + |Du|^2)^{1/n} \left[ (K^{1/n})_{\xi\xi}\sqrt{1 + |Du|^2} + 2(K^{1/n})_{\xi} \frac{n+2}{n} \frac{u_mu_{m\xi}}{\sqrt{1 + |Du|^2}} + K^{1/n} \frac{n+2}{n} \frac{u_mu_{m\xi\xi}}{\sqrt{1 + |Du|^2}} \right] \, ,
	\] and so it suffices to employ the barrier \[
	w = \frac{u_{\xi\xi}}{M - u} + a_ku_k + b + M|x|^2 + N|Du|^2 \, ,
	\] where $N$ is sufficiently large to deal with the remaining second derivative terms arising from $\psi$.
\end{proof}


\subsection{Boundary Conditions}\label{sec-PGCBC}

Consider $u$ as the graph of a surface $\Sigma$ with coordinates $(x,u(x))\in\R^{n+1}$, $x\in\Omega$, $u \leq 0$, and downward-pointing normal $\normal_{\Sigma}$. Take $e_{n+1}$ to be the unit vector $(0,...,0,1) \in \R^{n+1}$, $\alpha$ to be the angle between $(\nu,0)$ and $\normal_{\Sigma}$ on $\Gamma^-$, and $\beta$ to be the angle between $\normal_{\Sigma}$ and $-e_{n+1}$ on $\Gamma^-$. Impose the inner boundary condition \[
\frac{\cos\alpha}{\cos\beta} \equiv \chi = \gamma_0u + \phi \, .
\] Since \[
\frac{\cos\alpha}{\cos\beta} = \frac{(\nu,0)\cdot\normal_{\Sigma}}{-e_{n+1}\cdot\normal_{\Sigma}} = \frac{(\nu,0)\cdot(Du,-1)}{-e_{n+1}\cdot(Du,-1)} = u_{\nu} \, ,
\] we can now solve the equation of prescribed Gauss curvature \eqref{PrescribedGauss} to obtain a $C^{2,\alpha}$ solution $u$ with this Neumann boundary condition on $\Gamma^-$, along with a homogeneous boundary condition on $\Gamma^+$.


\section{Parabolic Flows}\label{sec-Parabolic}

In this section, we consider a parabolic analogue of the Monge-Amp\`ere equation given by \eqref{TheFlowEquation} on the annular cylinder $\Omega_T$, where $0 < T < \infty$, $u_0\in C^\infty(\overline{\Omega})$ is a strictly convex solution of $\det[D^2u_0] = \psi^n(x,u_0,Du_0)$ on $\Omega$, $\vartheta'(t) < 0$ for all $t$, and for each fixed $x\in\Gamma^-$ and all $t$,  $\phi_t(x,t) > 0$. We assume additionally that $\psi_z(x,z,p) \geq 0$. The equation is related to the inverse Gauss curvature flow:
\begin{example}
	Let $X$ be a hypersurface in $\R^{n+1}$ expanding under the inverse Gauss curvature flow,
	\begin{equation}\label{IGCF}
	\frac{\d X}{\d t} = \frac{1}{K} \normal \, ,
	\end{equation}
	where $K$ is the Gauss curvature of $X$ and $\normal$ is the outward-pointing unit normal vector to $X$. Suppose that $X$ may be written as a graph evolving over the annular domain $\Omega$. Then $X(t) = (x,u(x,t))$, $(x,t)\in\Omega_T$, $u \leq 0$. By taking the inner product of \eqref{IGCF} with $\normal$ and using the formula for the Gauss curvature of a graph, \[
	- \frac{u_t}{\sqrt{1 + |Du|^2}} = \frac{(1 + |Du|^2)^{(n+2)/2}}{\det[D^2u]} \, .
	\] Upon rearranging, this is equivalent to \eqref{TheFlowEquation} with $\psi^n(x,u,Du) = (1 + |Du|^2)^{(n+3)/2}$. Incidentally, this $\psi$ satisfies Condition \ref{StructureConditions}, although we shall see that this is not necessary for this type of parabolic equation. It must also be commented at this point that in order to obtain a global gradient estimate for the above equation, we would need to modify it so that the right-hand side tends to zero as $x$ approaches $\Gamma^+$, for instance by multiplying $1/K$ with $\dist(x,\Gamma^+)$.
\end{example}


\subsection{Upper and Lower Estimates for $u_t$}\label{sec-Flowut}

\begin{proposition}
	Suppose $u\in C^2(\Omega_T)\cap C^1(\overline{\Omega_T})$ is a strictly convex solution of \eqref{TheFlowEquation} with $\psi = \psi(x,u,Du)$. Then there exist constants $C_T,C^T$ depending on $T,\vartheta,\gamma_0$ and $\phi$ such that
	\begin{equation}\label{utBound}
	C_T \leq \sup_{\overline{\Omega}\times[0,T]} |u_t| \leq C^T \eqdef \max\left\{ 1 \, , \, \sup_{t\in(0,T)}|\vartheta'(t)| \, , \, \sup_{t\in(0,T)} \frac{\phi_t(\cdot,t)}{\gamma_0} \right\} \, .
	\end{equation}
\end{proposition}

\begin{proof}
	Taking the logarithmic derivative of \eqref{TheFlowEquation} gives us that $v = - u_t$ satisfies \[
	Lv = - \frac{\d v}{\d t} + v u^{ij}\frac{\d^2 v}{\d x_i \d x_j} - v \frac{n\psi_{p_i}}{\psi} \frac{\d v}{\d x_i} - \frac{n\psi_z}{\psi}v^2 = 0 \, .
	\] On the boundary, by the definition of $u_0$ and by differentiating the boundary conditions we have
	\begin{subequations}
		\begin{align}
		v = 1 & \quad \text{ on } \Omega\times\{0\} \, , \label{tBCBase} \\
		v = -\vartheta'(t) & \quad \text{ on } \Gamma^+\times(0,T) \, , \label{tBCG+} \\
		v_\nu - \gamma_0 v = - \phi_t(\cdot,t) & \quad \text{ on } \Gamma^-\times(0,T) \, . \label{tBCG-}
		\end{align}
	\end{subequations}
	Since $\psi_z , v \geq 0$, these boundary conditions and Theorem 2.4 of \cite{Lieberman} now imply $v \leq C^T$ as defined.
	
	For the lower estimate, assume for contradiction that there exists a point $(x_0,t_0) \in \overline{\Omega_T}$ such that $u_t(x_0,t_0) = 0$. Thanks to the upper estimate $C^T$ and $Lu_t = 0$, the strong parabolic maximum principle (Theorem 2.7 of \cite{Lieberman}) tells us that $(x_0,t_0) \in \d\Omega\times(0,T)\cup\Omega\times\{0\}$. From \eqref{tBCBase}, $t_0 \neq 0$. For $x_0 \in \Gamma^+$, \eqref{tBCG+} implies $u_t(x_0,\cdot) \leq \vartheta'(\cdot) < 0$; thus we are forced to conclude that $x_0 \in \Gamma^-$. This is a maximum point, so $\d_\nu u_t = u_{t\nu} \leq 0$. However, by \eqref{tBCG-}, $u_{t\nu}(x_0,t_0) = \phi_t(x_0,t_0) > 0$, which finally gives us our contradiction. Therefore there exists $C_T > 0$ such that $u_t \leq -C_T$, proving the estimate \eqref{utBound}.
\end{proof}


\subsection{The $C^0$ estimate}\label{sec-FlowC0}

\begin{proposition}\label{FlowC0}
	Suppose $u\in C^2(\Omega_T)\cap C^1(\overline{\Omega_T})$ is a strictly convex solution of \eqref{TheFlowEquation}. Then for all $(x,t)\in\overline{\Omega_T}$, $u$ is bounded with \[
	\sup_{\overline{\Omega_T}}|u| \leq C_0^T \eqdef TC^T + \sup_{\overline{\Omega}} |u_0| \, .
	\]
\end{proposition}

\begin{proof}
	Integrate \eqref{utBound} with respect to $t$.
\end{proof}


\subsection{The $C^1$ estimate}\label{sec-FlowC1}

\begin{proposition}\label{GradientFlow}
	Suppose $u \in C^2(\Omega_T)\cap C^1(\overline{\Omega_T})$ is a strictly convex solution of \eqref{TheFlowEquation} with $\psi = \psi(x,u)$. Then
	\begin{equation}\label{GradientFlowBound}
	\sup_{\overline{\Omega}} |Du| \leq C_1^T \eqdef \max \left\{ \frac{C_0^T}{1 - \max_{\Gamma^-}|\rho|} , \left( \frac{\|\psi\|_\infty}{\lambda_\emph{min}} \right)^{\frac{n}{n+1}} \right\} \max_{\Gamma^+} |D \rho| \, ,
	\end{equation} where $\rho$ is a strictly convex defining function for $\Gamma^+$ such that both $\rho(x) = 0$ and $0 < |D \rho(x)| \leq 1$ on $\Gamma^+$, $\max_{\Gamma^-}|\rho| < 1$, $\lambda_\emph{min}$ is the minimum eigenvalue of $D^2 \rho$, and $C_0^T$ is defined in Proposition \ref{FlowC0}.
\end{proposition}

\begin{proof}
	Since $u$ is convex, the maximum of $|Du|$ occurs on the outer boundary $\Gamma^+$. From the Dirichlet condition, the tangential derivatives of $u$ are zero, and so $|Du| = -u_\normal$, where $\normal$ is the inward-pointing unit normal field on $\Gamma^+$.
	
	Let $\rho$ be a strictly convex defining function for $\Gamma^+$ such that both $\rho(x) = 0$ and $0 < |D \rho(x)| \leq 1$ on $\Gamma^+$. Consider the function \[
	v = u - K(\rho - e^t)
	\] on $\Omega$ for some constant $K$ to be chosen later. Recalling that $u$ and $\rho$ are both negative on $\Omega$ and equal on $\Gamma^+$, we claim that for large enough $K$, $v$ cannot have a local minimum inside $\Omega$. Indeed, if $v$ has a minimum at $x_0 \in \Omega$, then at $x_0$, \[
	0 = v_t = u_t + Ke^t \, .
	\] Applying this to obtain \[
	\sum_{i=1}^{n} u^{ii} \geq \frac{n}{\det[D^2u]^{1/n}} = \frac{n (-u_t)^{1/n}}{\|\psi\|_\infty} = \frac{n}{\|\psi\|_\infty}\left( \frac{K(-u_t)}{2} \right)^{1/n} \, ,
	\] we have \[
	0 \leq u^{ij}v_{ij} = u^{ij}u_{ij} - Ku^{ij}\rho_{ij} \leq n - K\lambda_\text{min} \sum u^{ii} \leq n - K\lambda_\text{min} \frac{n(Ke^t)^{1/n}}{\|\psi\|_\infty} < 0 \, ,
	\] provided we choose \[
	K > \left( \frac{\|\psi\|_\infty}{\lambda_\text{min}} \right)^{\frac{n}{n+1}} \, .
	\] Here $\lambda_\text{min} > 0$ is the minimal eigenvalue of $D^2 \rho$, which is strictly positive by strict convexity. If additionally we impose \[
	K > \frac{C_0^T}{1-\max_{\Gamma^-}|\rho|} \, ,
	\] then $v > 0$ on $\Gamma^-$. Thus the minimum occurs on $\Gamma^+$. Therefore we have $v_{\normal} \geq 0$, where $\normal$ is the inward-pointing unit normal field on $\Gamma^+$, and so \[
	|D u| = -u_\normal \leq -K\rho_\normal = K|D\rho| \, ,
	\] from which the bound \eqref{GradientFlowBound} follows.
\end{proof}

\begin{remark}
	When $\psi = \psi(x,u,Du)$, we must impose Structure condition \ref{StructureConditionFlow} on $\psi$. The global gradient estimate then follows from the argument prior to Theorem 15.25 in Chapter 15 of \cite{Lieberman}.
\end{remark}


\subsection{The $C^2$ estimate}\label{sec-FlowC2}

\begin{proposition}\label{FlowC2}
	Suppose $u \in C^4(\Omega_T)\cap C^3(\overline{\Omega_T})$ is a strictly convex solution of \eqref{TheFlowEquation}. Then there exists a constant $C_2^T = C_2^T(n,\Omega,\Gamma^\pm,\psi,\vartheta,\phi)$ such that \[
	\sup_{\overline{\Omega_T}}|D^2 u| \leq C_2^T \, ,
	\] provided that there exists a subsolution satisfying Condition \ref{FlowSubsolutionCondition} and either $\psi = \psi(x,u)$ and \[
	2\max_{\substack{i = 1,...,n-1 \\ \bar{x} \in\Gamma^-}} \kappa_i(\bar{x}) < \gamma_0 + \max \left\{ 0 , \min_{\Gamma^-}\frac{\gamma_0 u + \phi}{M-u} \right\} \, ,
	\] where $\kappa_i(\bar{x})$ is the $i^{\text{th}}$ principal curvature of $\Gamma^-$ at $\bar{x}$, and $M = M(n,\Omega,\Gamma^\pm,\psi,\phi)$ is a constant, or $\psi = \psi(x,u,Du)$ and \[
	2\max_{\substack{i = 1,...,n-1 \\ \bar{x} \in\Gamma^-}} \kappa_i(\bar{x}) + \tilde{C} < \gamma_0 + \max \left\{ 0 , \min_{\Gamma^-}\frac{\gamma_0 u + \phi}{\tilde{M} + (1-N^4)u} \right\} \, ,
	\] where $\kappa_i(\bar{x})$ is the $i^{\text{th}}$ principal curvature of $\Gamma^-$ at $\bar{x}$, $N = N(n,\Omega,\psi,\phi)$, $\tilde{C} = \tilde{C}(n,\Omega,\Gamma^\pm,\psi,\phi,N)$ and $\tilde{M} = \tilde{M}(n,\Omega,\Gamma^\pm,\psi,\phi)$ are constants.
\end{proposition}

\begin{proof}
	As with the $C^2$ estimates in previous sections, the full $C^2$ estimate reduces to boundary estimates. This follows from the argument preceding Theorem 15.22 in \cite{Lieberman}. \\
	
	We first assume $\psi = \psi(x,u)$. Consider the generalised auxiliary function \[
	\hat{w} = w - Mt = g(u)u_{\xi\xi} + a_k u_k + b + M(|x|^2 - t) \, ,
	\] where $\xi$ is an arbitrary direction, $g(u)$ is a nonnegative function to be chosen, $M$ is a constant also to be chosen, and $a_k$ and $b$ are as in \cite{LTU}:
	\begin{gather*}
	a_k = 2(\xi\cdot \nu)(\chi_u\xi_k'-\xi_k'D_i\nu_k) \, , \\
	b = 2(\xi\cdot \nu)\xi_k'\chi_{x_k} \, ,
	\end{gather*}
	$\xi'=\xi-(\xi\cdot\nu)\nu$ with $\nu$ a $C^{2,1}(\Omega)$ extension of the inner unit normal field on $\Gamma^-$.
	
	Suppose for contradiction that $\hat{w}$ has a local maximum at $(x_0,t_0) \in \Omega\times(0,T)$. Then at $x_0$ we have
	\begin{align}
	& \hat{w}_t = g' u_t u_{\xi\xi} + g u_{t\xi\xi} + a_k u_{kt} + b_t - M = 0 \, , \label{FlowC2 wt} \\
	& \hat{w}_i = g' u_i u_{\xi\xi} + g u_{i\xi\xi} + a_{k,i}u_k + a_ku_{ki} + b_i + 2Mx_i = 0 \, , \label{FlowC2 wi} \\
	\begin{split}\label{FlowC2 wij}
	& \hat{w}_{ij} = g'' u_i u_j u_{\xi\xi} + g' u_{ij}u_{\xi\xi} + g' u_i u_{j\xi\xi} + g' u_j u_{i\xi\xi} + g u_{ij\xi\xi} \\
	& \qquad \ \ + a_{k,ij}u_k + a_{k,i}u_{kj} + a_{k,j}u_{ki} + a_k u_{kij} + b_{ij} + 2M\delta_{ij} \leq 0 \, .
	\end{split}
	\end{align}
	
	Applying Jacobi's formula to \eqref{TheFlowEquation}, we also have for $F^{ij} = u^{ij}$,
	\begin{gather}
	\frac{u_{t\xi}}{u_t} + F^{ij}u_{ij\xi} = \ln(\psi^n)_\xi \, , \label{FlowWhy} \\
	\frac{u_{t\xi\xi}}{u_t} - \left( \frac{u_{t\xi}}{u_t} \right)^2 + F^{ij}u_{ij\xi\xi} = \ln(\psi^n)_{\xi\xi} - F^{ij}_\xi u_{ij\xi} = \ln(\psi^n)_{\xi\xi} + F^{ik}F^{jl}u_{ij\xi}u_{kl\xi} \, . \label{FlowFijwijxixi}
	\end{gather}
	From $b$, define $\overline{b_k} = 2(\xi\cdot \nu)\xi_k'$ as before. With the above notation, we have from \eqref{FlowWhy}
	\begin{gather*}
	\frac{b_t}{u_t} + F^{ij}b_{ij} = \frac{\gamma_0 \overline{b_k} (u_{kt} + \phi_{kt})}{u_t} + F^{ij}\Bigl[ \overline{b_k}_{,ij}(\gamma_0u_k + \phi_k) + 2\overline{b_k}_{,i} \phi_{jk} + \overline{b_k}\phi_{ijk} \Bigr] + 2\gamma_0 \overline{b_k}_{,k} + \gamma_0 \overline{b_k} F^{ij}u_{ijk} \\
	= \frac{\gamma_0 \overline{b_k} \phi_{kt}}{u_t} + F^{ij}\Bigl[ \overline{b_k}_{,ij}(\gamma_0u_k + \phi_k) + 2\overline{b_k}_{,i} \phi_{jk} + \overline{b_k}\phi_{ijk} \Bigr] + 2\gamma_0 \overline{b_k}_{,k} + \gamma_0 \overline{b_k} \ln(\psi^n)_k \, ,
	\end{gather*}
	and therefore by \eqref{FlowC2 wi} and \eqref{FlowC2 wij},
	\begin{gather*}
	F^{ij}\hat{w}_{ij} = \left[ ng' + \left( g'' - \frac{2(g')^2}{g} \right)F^{ij}u_i u_j \right]u_{\xi\xi} - \frac{2g'}{g}F^{ij}u_i \left( u_k a_{k,j} + b_j + 2Mx_j \right) - \frac{2g'}{g} a_k u_k \\
	+ g\left[ \ln(\psi^n)_{\xi\xi} + F^{ij}u_{jk\xi}F^{kl}u_{li\xi} + \left( \frac{u_{t\xi}}{u_t} \right)^2 \right] + \frac{1}{u_t}\left( g'u_t u_{\xi\xi} + a_k u_{kt} + b_t - M \right) + F^{ij}a_{k,ij}u_k \\
	+ \sum a_{k,k} + a_k\left( \ln(\psi^n)_k - \frac{u_{kt}}{u_t} \right) + F^{ij}b_{ij} + 2M\mathcal{T} \\
	= \left[ (n+1)g' + \left( g'' - \frac{2(g')^2}{g} \right)F^{ij}u_i u_j \right]u_{\xi\xi} - \frac{2g'}{g}F^{ij}u_i \left( u_k a_{k,j} + b_j + 2Mx_j \right) - \frac{2g'}{g} a_k u_k \\
	+ g\left[ \ln(\psi^n)_{\xi\xi} + F^{ij}u_{jk\xi}F^{kl}u_{li\xi} + \left( \frac{u_{t\xi}}{u_t} \right)^2 \right] + F^{ij}\Bigl[ a_{k,ij}u_k + \overline{b_k}_{,ij}(\gamma_0 u_k + \phi_k) + 2\overline{b_k}_{,i} \phi_{kj} + \overline{b_k}\phi_{ijk} \Bigr] \\ + (a_k + \overline{b_k} \gamma_0) \ln(\psi^n)_k - \frac{1}{u_t} + 2\sum (a_{k,k} + \gamma_0 \overline{b_k}_{,k}) + 2M\mathcal{T} + \frac{\gamma_0 \overline{b_k} \phi_{kt}}{u_t} \, .
	\end{gather*}
	Here we used \eqref{FlowC2 wt} to remove the $u_{t\xi\xi}$ term arising from \eqref{FlowFijwijxixi}. We also deal with the terms arising from $\ln(\psi^n)_{\xi\xi}$ in the same way as in Remark \ref{psi(x,u)}. We can now proceed exactly as we did from Lemma \ref{Fijwij} to obtain the result of Proposition \ref{FlowC2}, since all the extra terms are positive or bounded. In particular, by Proposition \ref{utBound}, $1/|u_t| \leq 1/C_T < \infty$. \\
	
	To obtain the tangential $C^2$ estimate when $\psi = \psi(x,u,Du)$, define $\overline{w} = e^{N|Du|^2}\hat{w}$, with $\hat{w}$ as in the previous subsection. This time, at a critical point, \[
	\frac{\overline{w}_t}{e^{N|Du|^2}} = 0 = 2N\hat{w}u_ku_{kt} + \hat{w}_t \, ,
	\] and therefore the extra term $2N\hat{w}u_ku_{kt}/u_t$ from $2NF^{ij}u_{ijk}u_k$ is cancelled when this is substituted into $F^{ij}\overline{w}_{ij}$. Thus the calculation runs exactly the same result as Proposition \ref{th-C2Du}, producing a bound on $|D^2 u|$ under the subsolution condition \ref{FlowSubsolutionCondition}. Note that by the parabolic maximum principle, any subsolution satisfying \eqref{TheFlowEquationSubsolution} remains a subsolution over the course of the flow. \\
	
	As in Section \ref{sec-existence}, the mixed tangential-normal derivative bounds on $\Gamma^-$ follow directly from the Neumann condition and the $C^1$ estimate in the same way as in \eqref{uxinu}.
	
	For the double normal estimate on $\Gamma^-$, we make use of \eqref{utBound}, which tells us that $u_t$ is bounded away from $0$. Therefore we may bound $u_{\nu\nu}$ above using the same method as that of Proposition \ref{Double Normal Estimate}, assuming the existence of a subsolution satisfying the extra conditions at all times $t\in[0,T]$.
	
	Finally, the estimates for $D^2 u$ on the outer boundary $\Gamma^+$ follow from the methods in Chapter 15 of \cite{Lieberman}, since our subsolution in Condition \ref{FlowSubsolutionCondition} is assumed to be strict. This completes the proof of the full $C^2$ estimate in Proposition \ref{FlowC2}.
\end{proof}

We can now employ the parabolic Krylov-Safonov $C^{2,\alpha}$ estimates of \cite{Lieberman} and the method of continuity to obtain a $C^{2,\alpha}$ solution of \eqref{TheFlowEquation}.

\end{document}